\documentclass{article}
\usepackage[hidelinks]{hyperref}
\usepackage{amssymb}
\usepackage{amsmath}
\numberwithin{equation}{section}
\usepackage{amsthm}
\usepackage{bookmark}
\usepackage{setspace}
\usepackage{mathtools}
\usepackage{tikz-cd}
\usepackage{manfnt}
\usepackage{fullpage}
\usepackage{kpfonts}
\newtheorem{prop}{Proposition}
\newtheorem*{nothm}{Theorem}
\newtheorem{thm}[prop]{Theorem}
\newtheorem{lem}[prop]{Lemma}
\newtheorem{cor}[prop]{Corollary}

\theoremstyle{definition}
\newtheorem{de}{Definition}

\newcommand{\1}{\boldsymbol{1}}

\newcommand{\aut}{\mathrm{Aut}}

\newcommand{\bk}{\vspace{5mm}\begin{center}
		$\ast\,\ast\,\ast$
	\end{center}\vspace{5mm}}

\newcommand{\C}{\mathbf{C}}
\newcommand{\cali}[1]{\mathcal{#1}}
\newcommand{\chx}[1]{\langle #1\rangle}

\newcommand{\G}{\mathbf{G}}

\newcommand{\got}[1]{\mathfrak{#1}}

\newcommand{\Hom}{\mathrm{Hom}}

\newcommand{\ide}{\mathbf{A}^{\times}}
\newcommand{\ilim}{\varinjlim}

\newcommand{\norm}[1]{\|#1\|}

\newcommand{\pair}[2]{\langle #1, #2 \rangle}

\newcommand{\plim}{\varprojlim}

\newcommand{\Q}{\mathbf{Q}}

\newcommand{\R}{\mathbf{R}}
\newcommand{\re}{\mathrm{Re}}

\newcommand{\Scr}[1]{\mathscr{#1}}

\newcommand{\spec}{\mathrm{Spec}}

\newcommand{\Z}{\mathbf{Z}}

\newcommand{\adefp}{\mathbf{A}_{f}^{p}}
\newcommand{\bmu}{\boldsymbol{\mu}}
\newcommand{\cts}{\mathrm{cts}}

\newcommand{\nm}{\mathrm{Nm}}

\newcommand{\chihaar}{\chi_{\mathrm{Haar}}}
\newcommand{\comment}[1]{}
\newcommand{\GrN}[2]{\underline{\got{l}^{-#1}\got{o}_K/\got{l}^{#2}}}
\newcommand{\grN}[2]{\got{l}^{-#1}\got{o}_K/\got{l}^{#2}}
\newcommand{\GrNvee}[2]{\underline{\got{l}^{-#1}\got{d}^{-1}/\got{l}^{#2}\got{d}^{-1}}}
\newcommand{\grNvee}[2]{\got{l}^{-#1}\got{d}^{-1}/\got{l}^{#2}\got{d}^{-1}}
\newcommand{\idefp}{(\mathbf{A}_{f}^{p})^\times}
\newcommand{\okt}{\mathfrak{o}_K^\times}
\newcommand{\ord}{\mathrm{ord}}
\newcommand{\op}{\got{o}_p}
\newcommand{\opn}{\got{o}_p/p^n}
\newcommand{\tr}{\mathrm{tr}}

\newcommand\blfootnote[1]{%
	\begingroup
	\renewcommand\thefootnote{}\footnote{#1}%
	\addtocounter{footnote}{-1}%
	\endgroup
}

\title{Note on $p$-adic Local Functional Equation}
\author{Luochen Zhao}
\date{}

\begin{document}
	\maketitle
	
	\blfootnote{\textit{Date}: Jun 5, 2022}
	
	\begin{abstract}
		Given primes $\ell\ne p$, we record here a $p$-adic valued Fourier theory on a local field over $\mathbf{Q}_\ell$, which is developed under the perspective of group schemes. As an application, by substituting rigid analysis for complex analysis, it leads naturally to the $p$-adic local functional equation at $\ell$, which strongly resembles the complex one in Tate's thesis.
	\end{abstract}
	
	\section{Introduction}
	
	Let $K$ be a local field of characteristic 0, $f$ be a Schwartz-Bruhat function on $K$, and $c:K^\times \to \C^\times$ be a unitary continuous character. A classical result of Tate states that there is a local functional equation \cite[Theorem 2.4.1]{Ta50}
	\begin{align}
		\label{equation.local-fe-complex}
		\int_{K}f(x)c(x)|x|^s \frac{dx}{|x|} = \rho_{\C}(c|\cdot|^s) \int_{K}\hat{f}(x)c^{-1}(x)|x|^{1-s}\frac{dx}{|x|}
	\end{align}
	
	for $s\in \C$ with $0< \re(s)<1$. Here, $dx$ is an additive Haar measure on $K$, $|\cdot|$ is the character given by $d(ax) = |a|dx$ for all $a\in K^\times$, $\hat{f}$ is the Fourier transform of $f$, and $\rho_{\C}(c|\cdot|^s)$ is a meromorphic function on $\C$ independent of $f$. As an immediate conseqeunce, Tate deduced that the zeta integral $\int_{K}f(x)c(x)|x|^{s-1}dx$, a priori holomorphic on the half plane $\re(s)>0$, admits a meromorphic continuation to all of $\C$. 
	
	\vspace{3mm}
	
	The main aim of this note is to establish a $p$-adic valued analogue to \eqref{equation.local-fe-complex} when $K$ is a finite extension of $\Q_\ell$, where $\ell\ne p$. The organization is as follows: In §2, by employing Cartier duality in the manner of Katz, we shall establish the $p$-adic Fourier transform on $K$, which recovers a classical result of Schikhof. In §3, using elementary rigid analysis, we will discuss $p$-adic zeta integrals and their analytic nature. Eventually, following Tate's original argument, we will establish the $p$-adic local functional equation in this rigid analytic setting. For the benefit of the reader, we have included Appendix A, which presents some motivational computations for the $p$-adic local functional equation. In Appendix B, we briefly explain how the ideas developed in the main sections might shed some light on a hypothetical functional equation of Kubota-Leopoldt $p$-adic $L$-functions.
	
	\vspace{3mm}
	
	We now introduce some notation so as to state our main results below. With a fixed local field $K/\Q_\ell$ where $\ell\ne p$, let $\got{o}_K$ be the valuation ring of $K$, $\got{l}$ be its maximal ideal, $\pi\in \got{l}$ be a uniformizer, and write $q = \#\got{o}_K/(\pi)$. Also, denote by $\got{d} = (\pi^{\delta})$ the different ideal of $\got{o}_K/\Z_\ell$. Let $\C_p$ be the completion of a fixed algebraic closure of $\Q_p$, and denote by $|\cdot|_p$ the absolute value on $\C_p$ extending that on $\Z_p$, so $|p|_p = 1/p$. Let $\op$ be the ring of integers in $\C_p$. Finally, if $E$ is a subset of $K$, we denote by $\1_{E}$ the characteristic function $\1_E: K \to \{0,1\}$, i.e., $\1_E(x) = 1$ if $x\in E$ and $\1_E(x) = 0$ otherwise. 
	\begin{nothm}[A]
		Let $\cali{C}(K,\got{o}_p)$ be the Fr\'echet space of continuous functions from $K$ to $\got{o}_p$, and let $\cali{D}(K,\got{o}_p)$ be the continuous linear dual $\Hom^{\rm cts}_{\got{o}_p}(\cali{C}(K,\got{o}_p),\got{o}_p)$. By choosing a compatible system of primitive $\ell$-power roots of unity $\zeta = (\zeta_n)_{n\ge 0} \in \plim_{n\ge 0}\mu_{\ell^n}$, we have an embedding
		\begin{align*}
			\cali{D}(K,\got{o}_p) \hookrightarrow \cali{C}(K,\got{o}_p) \qquad \mu\mapsto \hat{\mu}_{\zeta},
		\end{align*}
		
		which identifies $\cali{D}(K,\got{o}_p)$ as uniformly continuous functions on $K$. Furthermore, if we let $dx$ denote the linear functional on compactly supported locally constant functions on $K$ defined by $dx(\1_{a+\got{l}^n}) = q^{-\delta/2-n}$, then for any $f\in \cali{C}(K,\got{o}_p)$ such that $\lim_{x\to \infty} f(x) = 0$, we have $fdx \in \cali{D}(K,\got{o}_p)$, and
		\begin{align*}
			(fdx)^{\wedge}_{\zeta}(\xi) = \int_{K} f(x)\zeta^{\tr(x\xi)} dx.
		\end{align*}
		
		Here, $\tr:K\to \Q_\ell$ is the trace map, and $\zeta^x$ is the function on $\Q_\ell$ such that if $x = \ell^t u$ with $u\in \Z_\ell$, $\zeta^x = \zeta_{-t}^u$ (we write $\zeta_n = 1$ for $n< 0$).
	\end{nothm}
	
	\begin{nothm}[B]
		Let $f\in \cali{C}(K,\got{o}_p)$ be such that $0 = \lim_{x\to \infty}f(x) = \lim_{x\to \infty}\hat{f}(x) = f(0) = \hat{f}(0)$. Let $\tilde{\chi}: K^\times \to \C_p^\times$ be a continuous character with $\tilde{\chi}(\pi) = 1$, and for any $\lambda\in \C_p^\times$, let $\chi_\lambda:K^\times \to \C_p^\times$ be the unramified character with $\chi(\pi) = \lambda$. Then the zeta integral $Z(f,\tilde{\chi}\chi_\lambda) = \int_{K} f(x)\tilde{\chi}\chi_\lambda(x)dx$ is a rigid analytic function on the Laurent domain $\{\lambda \in \C_p^\times: |\lambda|_p = 1\}$, and satisfies a local functional equation:
		\begin{align}
			\label{equation.local-fe}
			\int_{K}f(x)\tilde{\chi}\chi_\lambda(x) dx = \rho(\tilde{\chi}\chi_\lambda)\int_{K}\hat{f}(x)\tilde{\chi}^{-1}\chi_{q\lambda^{-1}}(x) dx,
		\end{align}
		
		where $\rho(\tilde{\chi}\chi_\lambda)$ is a rigid meromorphic function on $\C_p^\times$ independent of $f$. Moreover, if both $f$ and $\hat{f}$ are of Schwartz class $c$ for some $c\ge 1$ (see Definition \ref{definition.schwartz-class}), then \eqref{equation.local-fe} implies that $Z(f,\tilde{\chi}\chi_\lambda)$ can be meromorphically continued to the annulus $\{\lambda\in \C_p^\times: c^{-1}\le |\lambda|_p\le c\}$.
	\end{nothm}
	
	We remark that our approach to Theorem (A), namely via group schemes, has the merit that it does not presume any knowledge of the complex Fourier transform (which is really a special case of Pontryagin duality), despite the resulting explicit formula for $(fdx)^\wedge_\zeta$ taking the same form as the complex one. In other words, it can be seen as an instance where Cartier and Pontryagin duality ``coincide''. On the other hand, Theorem (B) shows that, if we supplant complex analysis by rigid analysis, the $p$-adic valued locally theory bears an uncanny resemblance to the complex one presented in §2.4 of \cite{Ta50}.

	\section{Generalities on Fourier theory}
	In §2.1 we shall consider functions on $K/\Q_\ell$ from the viewpoint of analysis and discuss their linearization; in §2.2 we use the linearization to give an algebraic viewpoint, and thereby realize the statement ``$K$ is self dual'' as known in classical Fourier analysis. This eventually leads to our construction of the Fourier transform in §2.3. 
	
	\vspace{3mm}
	
	\textit{Remark on prior works.}
	We emphasize that the algebraization of ``distributions'' on an $\ell$-adic ring of integers is not new, and in fact can be found in the literature as early as 1960s, e.g., \cite{Am64} and \cite{MS74}, which treated mainly the $\ell=p$ case. The (formal) group scheme interpretation is prevalent in the works of Katz on $p$-adic $L$-functions and congruences of Bernoulli-Hurwitz numbers, see for example \cite{Ka81} and \cite{Ka82}. On the other hand, to the best knowledge of the author, the $\ell\ne p$ case has attracted little attention from number theorists (perhaps the most notable being \cite{Si87}), and is only treated in detail, with strong analysis flavor, in W.~H.~Schikhof's PhD thesis as recorded in Chapter 9 of \cite{vR78}. In particular, one should note that the main result of this section, Theorem \ref{thm.fourier-inversion}, is in fact due to Schikhof (\textit{cf}. Theorem 9.21 \textit{loc. cit.}), although the author was unaware of the work of Schikhof until he finished writing this note. Still, we have not found our treatment in terms of group schemes anywhere in the literature, and as emphasized in the introduction, our treatment will be independent of the classical Pontryagin duality, which plays a key role in Schikhof's Fourier theory. As such, we believe that our treatment might be more natural, and better suited for modern number theorists .
	
	\subsection{Linearization of function spaces}
	For topological spaces $X,Y$, we let $\cali{C}(X,Y)$ denote all continuous functions from $X$ to $Y$. For $\cali{C}(K,\got{o}_p)$, we can make it a Fr\'echet space: for each compact open subset $U$ of $K$, we can define a semi-norm $\norm{f}_U = \sup_{x\in U}|f(x)|_p$. It is easy to see that $(\cali{C}(U,\got{o}_p), \norm{\cdot}_U)$ is a complete normed $\op$-module. Since by restriction we have an isomorphism of $\got{o}_p$-modules, $\cali{C}(K,\got{o}_p)\simeq \plim_U \cali{C}(U,\got{o}_p)$, we may make $\cali{C}(K,\got{o}_p)$ Fr\'echet with $\{\norm{\cdot}_U\}_U$ being its topology-defining semi-norms. 
	
	\vspace{3mm}
	
	We note that modulo $p^n$, we have a natural identification $\cali{C}(\got{l}^{-r}\got{o}_K,\got{o}_p/p^n) \simeq \ilim_{N} \cali{C}(\got{l}^{-r}\got{o}_K/\got{l}^N, \got{o}_p/p^n)$, because any continuous function on $\got{l}^{-r}\got{o}_K$ valued in the discrete ring $\got{o}_p/p^n$ must be locally constant and $\got{l}^{-r}\got{o}_K$ is compact. Now let $r$ vary and replace the above family $\{U\}_{U: \rm open\ compact}$ by $\{\got{l}^{-r}\got{o}_K\}_{r\ge 0}$ and $\got{o}_p$ by $\got{o}_p/p^n$, we can prolongate the above identification of $\op/p^n$-modules: 
	\[
		\cali{C}(K,\got{o}_p/p^n) \simeq \plim_{r}\ilim_{N} 
		\cali{C}(\got{l}^{-r}\got{o}_K/\got{l}^N, \got{o}_p/p^n),
	\]
	
	where the projective limit is taken with respect to restrictions and the inductive limit is defined using inclusions. Inside this space, we have also the uniformly continuous functions given by inverting limits: $\ilim_{N}\plim_{r}\cali{C}(\got{l}^{-r}\got{o}_K/\got{l}^N,\got{o}_p/p^n) = \ilim_{N}\cali{C}(K/\got{l}^N,\got{o}_p/p^n)$. 
	
	\vspace{3mm}
	\textbf{Example.} The function $f = \sum_{n\ge 0} \1_{1/\ell^n + \ell^n\Z_\ell}$ is continuous, but not uniformly continuous.
	
	\vspace{3mm}
	The linearization above is useful in that it is also topological: 
	
	\begin{prop}
		\label{prop.topological-linearization}
		Let $p_{n,r}: \cali{C}(K,\got{o}_p) \to 
		\cali{C}(\got{l}^{-r}\got{o}_K,\got{o}_p/p^n)$ be the canonical projection by restriction and modulo $p^n$. Then the Fr\'echet space $\cali{C}(K,\got{o}_p)$ is linearly topologized by $\{\ker(p_{n,r})\}_{n,r}$; i.e., these kernels are $\got{o}_p$-modules, and they form a fundamental system at the zero function. Hence we have a topological linearization:
		\begin{align}
		\label{equation.linearization}
		\cali{C}(K,\op) \simeq \plim_{n}\plim_{r:\rm{res.}}\ilim_{N:\textrm{\rm{incl.}}}
		\cali{C}(\grN{r}{N},\opn)
		\end{align}
		where the limit over $r$ is with respect to restriction and the colimit over $N$ is with respect to inclusion (as functions on $\got{l}^{-r}\got{o}_K$).
	\end{prop}
	
	\begin{proof}
		First, by definition we have $\ker(p_{n,r}) = \{f\in \cali{C}(K,\got{o}_p) : f(\got{l}^{-r}\got{o}_K) \subseteq p^n\got{o}_p\}$. In particular, it is an $\got{o}_p$-module. Next, we note that under the semi-norm $\norm{\cdot}_{\got{l}^{-r}\got{o}_K}$, $\ker(p_{n,r})$ is nothing but $\{g:\norm{g}_{\got{l}^{-r}\got{o}_K}\le p^{-n}\}$. Therefore $\ker(p_{n,r})$ is open. Finally, to see that they form a fundamental system, we simply observe that the sub-family of semi-norms $\norm{\cdot}_{\got{l}^{-r}\got{o}_K}$, for $r\in \Z_{\ge 0}$, already give the Fr\'echet topology of $\cali{C}(K,\got{o}_p)$, which conludes that the family of open sets $\{\ker(p_{n,r})\}_{n,r}$ have enough candidates to become a fundamental system at 0.
	\end{proof}
	
	Next, we define the continuous linear dual $\cali{D}(K,\got{o}_p) = \Hom^{\mathrm{cts}}_{\got{o}_p}(\cali{C}(K,\got{o}_p), \got{o}_{p})$. A consequence of Proposition \ref{prop.topological-linearization} is the following: 
	\begin{cor}
		\label{cor.measure-linearization}
		There is a linearization:
		\begin{align}
		\label{equation.pre-linearization-measure}
		\cali{D}(K,\got{o}_p) \simeq \plim_{n}\ilim_{r}\plim_{N}
		\Hom_{\op}(\cali{C}(\got{l}^{-r}\got{o}_K/\got{l}^N, \opn),\opn),
		\end{align}
		where the colimit over $r$ is induced by restrictions and the limit over $N$ is induced by inclusions.
	\end{cor}
	\begin{proof}
		For any $L \in \cali{D}(K,\op)$ and $f\in \cali{C}(K,\op)$, to know the value $L(f)\in \op$ it is necessary and sufficient to know the values $L(f) \bmod p^n$ for all $n\ge 1$. So, we have an injection $\cali{D}(K,\got{o}_p) \hookrightarrow \plim_{n}\Hom^{\cts}_{\op}(\cali{C}(K,\op), \opn)$. This map is also surjective: For any $L=(L_n)_n \in \plim_{n}\Hom^{\cts}_{\op}(\cali{C}(K,\op), \opn)$, we can regard $L$ as a linear functional on $\cali{C}(K,\op)$ by $L(f)=\lim_n L_n(f)$. Furthermore, $L$ is continuous because for all $a\in \op$, $n\in \Z_{\ge 0}$ and $m\ge n$, $L^{-1}(a+ p^n\op) = L_m^{-1}(a+p^n\op)$ is an open subset of $\cali{C}(K,\op)$ by the continuity of $L_m$.
		
		\vspace{3mm}
		
		It suffices now to deal with $\Hom^{\cts}_{\op}(\cali{C}(K,\op),\opn)$. Note first that as $p^n \cali{C}(K,\op) = \cali{C}(K,p^n\op)$, we have $\Hom^{\cts}_{\op}(\cali{C}(K,\op),\opn) \simeq \Hom^{\cts}_{\op}(\cali{C}(K,\opn),\opn)$. Since $\opn$ is discrete, any linear functional belonging to this hom-set has an open kernel. By Proposition \ref{prop.topological-linearization} we see that the kernel contains some $\ker(p_{n',r})$, and thus contains $\ker(p_{n',r}) + p^n\cali{C}(K,\got{o}_p) \supseteq \ker(p_{n,r})$. Therefore, 
		\[
			\Hom^{\cts}_{\op}(\cali{C}(K,\op),\opn) = \ilim_{r}\Hom^{\cts}_{\op}(\cali{C}(\got{l}^{-r}\got{o}_K,\opn),\opn) = \ilim_{r}\plim_{N}\Hom^{\cts}_{\op}(\cali{C}(\grN{r}{N},\opn),\opn).
		\]
		
		We note finally that since $\cali{C}(\grN{r}{N},\opn)$ is discrete, we may drop the $\cts$ superscript.
	\end{proof}

	To unveil the nature of the space $\cali{D}(K,\op)$, we find it enlightening to think in terms of (commutative finite \'etale) group schemes. This will be carried out in the following section.

	\subsection{Group scheme theoretic interpretation}
	
	Consider the constant group scheme $\GrN{r}{N}$ defined over $\Z$, which is the sheafification of the functor $G_{r,N}(R) = \got{l}^{-r}\got{o}_K/\got{l}^N$ for any ring $R$. When $G$ is an affine group scheme over a ring $S$ and $R$ is an $S$-algebra, we shall write $\cali{A}_{G,R}$ for the affine algebra of $G\times_S R/\spec(R)$. It is easy to show that, as an algebra, $\cali{A}_{\GrN{r}{N},R} \simeq \oplus_{a\in \grN{r}{N}} R \delta_a$, where $\delta_a$'s are orthogonal idempotents and the identity is given by $\sum_a \delta_a$. Moreover, the comultiplication is given by the formula
	\[s: \delta_a \mapsto \sum_{b+c = a} \delta_b\otimes \delta_c.\] 
	
	This is relevant to us, because when we regard $\cali{C}(\grN{r}{N},\opn)$ as an $\opn$-algebra (its arithmetic coming from $\opn$), it is isomorphic to $\cali{A}_{\GrN{r}{N}}\otimes \opn$. For instance, the idempotent $\delta_a$ represents the characteristic function $\1_{a+\got{l}^N}(x)$ on $\grN{r}{N}$. 
	
	\vspace{3mm}
	Now we consider $\Hom_{\op}(\cali{C}(\grN{r}{N},\opn),\opn)$. If we identify $\cali{C}(\grN{r}{N},\opn)$ with $\cali{A}_{\GrN{r}{N}}\otimes{\opn}$, then we see the hom-set is naturally identified as the affine algebra of the Cartier dual of $\GrN{r}{N}$ over $\opn$. Temporarily we write $\G_{N,r} = (\GrN{r}{N})^\vee$, where $\vee$ stands for the Cartier dual. As such, by Corollary \ref{cor.measure-linearization}, we see immediately $\cali{D}(K,\op) \simeq \plim_{n}\ilim_{r}\plim_{N}\cali{A}_{\G_{N,r},\opn}$. By virtue of the next theorem, we may identify $\G_{N,r}$ with $\GrNvee{N}{r}$. Granting this, if we revert to analytic terms, we see that we have really shown ``measures on $K$ are (uniformly continuous) functions on $K$'' ------ or as we claimed earlier, ``$K$ is self-dual''. Strictly speaking, we also need to check ``functoriality'', namely the (co-)limits are exactly those we saw in the linearization of functions; this is taken care of by Proposition \ref{prop.functoriality} below. 
	
	\begin{thm}
		\label{thm.etalization}
		Let $\got{d}$ be the different ideal of $\got{o}_K/\Z_\ell$. Let $R$ be a ring in which $\ell$ is invertible, and which contains a primitive $\ell^{N+r}$-th root of unity $\zeta_{N,r}$. Then we have an isomorphism $\G_{N,r}\times_{\Z} R \simeq \underline{\got{l}^{-N}\got{d}^{-1}/\got{l}^{r}\got{d}^{-1}}$ depending on the choice of $\zeta_{N,r}$. 
	\end{thm}
	
	\textit{Remark.} This is where we find the perspective of group schemes extremely useful: we know an \'etale finite flat group scheme is constant over the algebraic closure. Since $\ell \ne p$ we can immediately apply this to our group scheme $\G_{N,r}$. In the proof however we shall give an explicit isomorphism, which is already hinted in §2.2~of Tate's thesis.
	
	\begin{proof}[Proof of Theorem \ref{thm.etalization}.]
		Our strategy is the following: First we construct a morphism from the constant presheaf of $\grNvee{N}{r}$ to $\G_{N,r}$. Then, since $\G_{N,r}$ is a sheaf, we extend the morphism to $\GrNvee{N}{r}$. Finally, we check the isomorphism at geometric points, which is enough since both groups are \'etale. 
		
		\vspace{3mm}
		\textit{Step 1.} Recall that if $G$ is a group scheme over $R$, the Cartier dual as a functor is given by $G^\vee(R') = \Hom_{R'-\textrm{GrpSch}}(G_{R'}, \G_{m,R'})$, an element of which is then an invertible element $x\in \cali{A}_{G,R'}$ such that $s(x) = x \otimes x$, where $s$ is the comultiplication.
		
		\vspace{3mm}
		
		\textit{Step 2.} Let $G_{N,r}$ be the constant functor such that $G_{N,r}(R') = \grNvee{N}{r}$ and $G(f)\colon G(S)\to G(T)$ is the identity map, for any $R', S, T$ over $R$ and $f:S/R\to T/R$. We shall construct a morphism $G_{N,r} \to \G_{N,r}$. Using the description of points in step 1, it suffices to construct functorial morphisms $\grNvee{N}{r} \to \Hom_{R'-\textrm{GrpSch}}(G_{R'}, \G_{m,R'})$. Concretely, for each $R'/R$, and each $b\in \grNvee{N}{r}$, we need to construct $x_b = x_{b,R'} \in (\oplus_{a\in \grN{r}{N}} R'\delta_a)^\times$ such that 
		\begin{itemize}
			\item $s(x_b) = x_b\otimes x_b$;
			\item $x_{a+b} = x_a x_b$;
			\item if $f: S\to T$ is an $R$-algebra map then $f(x_{b,S})=x_{b,T}$.
		\end{itemize}
		
		\vspace{3mm}
		\textit{Step 3.} Define a function $\zeta^x: \ell^{-N-r}\Z_\ell/\Z_\ell \to R^\times$ by $\zeta^x = \zeta_{N,r}^{[\ell^{N+r}]x}$ where for any $t>0$, $[\ell^t]: \ell^{-t}\Z_\ell/\Z_\ell \xrightarrow{\sim} \Z_\ell/\ell^t$ is the isomorphism sending $y$ to $\ell^t y$. We may compose $\zeta^x\circ \tr : \got{l}^{-N-r}\got{d}^{-1}/\got{d}^{-1} \to R^\times$, which makes sense since $\tr$ sends $\got{l}^{-N-r}\got{d}^{-1}/\got{d}^{-1}$ to $\ell^{-N-r}\Z_\ell/\Z_\ell$ (in general not surjective).
		
		\vspace{3mm}
		
		\textit{Step 4.} For $b \in \grNvee{N}{r}$, we write
		\[x_b = \sum_{a\in\grN{r}{N}} \zeta^{\tr(ab)}\delta_a.\]
		
		Seeing $\zeta^{\tr(ab)}\in R^\times$ and $\delta_a$'s are orthogonal idempotents, it is clear that $x_b\in (\oplus_{a\in \grN{r}{N}} R\delta_a)^\times$. Moreover, we verify:
		\begin{align*}
		\begin{split}
		s(x_b) &= \sum_{a\in\grN{r}{N}}\zeta^{\tr(ab)}\sum_{u+v =a}\delta_u\otimes\delta_v\\
		&= \sum_{u,v}\zeta^{\tr(ub)}\delta_u\otimes \zeta^{\tr(vb)}\delta_v\\
		&= x_b\otimes x_b,
		\end{split}
		\end{align*}
		and
		\[x_b x_c = \sum_{a\in\grN{r}{N}} \zeta^{\tr(ab)+\tr(ac)}\delta_a 
		= x_{b+c}.\]
		
		Finally, since by definition $x_b$ has coefficients in $R$, the functorial requirement is automatically satisfied.
		
		\vspace{3mm}
		\textit{Step 5.} We have constructed a natural morphism $G_{N,r} \to \G_{N,r}$. Hence by the universal property, it extends to $\GrNvee{N}{r} \to \G_{N,r}$. Since $\ell^{-1}\in R$, we know that both of the $\ell$-groups, $\grNvee{N}{r}$ and $\G_{N,r}$, are \'etale. Thus, the isomorphism can be checked by checking only Galois modules $\GrNvee{N}{r}(k) \to \G_{N,r}(k)$ for any $R\to k$ geometric (namely $k$ is an algebraically closed field). To do this, we need to show that any $x\in \Hom_{k-\textrm{GrpSch}}(G_{k},\G_{m,k})$ is of the form $x_b$ for a unique $b\in \grNvee{N}{r}$ (as such, both $\GrNvee{N}{r}(k)$ and $\G_{N,r}(k)$ are acted trivially by the Galois/fundamental group of $k/R$).
		
		\vspace{3mm}
		\textit{Step 6.} Let $x\in \Hom_{k-\textrm{GrpSch}}(G_k,\G_{m,k})$, and we may write $x = \sum_{a\in\grN{r}{N}}z(a)\delta_a$ where $z$ is valued in $k^\times$. Since $s(x) = x\otimes x$, expand it and we have
		\[\sum_{a\in\grN{r}{N}}z(a)\sum_{u+v =a}\delta_u\otimes \delta_v = \sum_{a\in\grN{r}{N}}z(a)\delta_a \otimes \sum_{b\in\grN{r}{N}}z(b)\delta_b.\]
		
		Comparing the coefficients of $\delta_a\otimes \delta_b$, we find that $z(a+b) = z(a)z(b)$. Therefore, we have shown that any $x$ is coming from some $z \in \Hom(\grN{r}{N}, k^\times)$, and it is easy to work out, conversely, any $z \in \Hom(\grN{r}{N}, k^\times)$ gives an $x\in \Hom_{k-\textrm{GrpSch}}(G_k,\G_{m,k})$ by the above formula. By using the following lemma we may then conclude $x = x_b$ for some unique $b \in \grNvee{N}{r}$. Hence, by step 5 the proof is finished.
	\end{proof}

	\begin{lem}
		\label{lem.pairing-perfect}
		Let $R$ be a ring containing a primitive $\ell^{N+r}$-th root of unity $\zeta_{N,r}$. The bilinear pairing $\grN{r}{N} \otimes \grNvee{N}{r} \to R^\times$ given by $a\otimes b \mapsto \zeta^{\tr(ab)}$ is perfect; i.e., it induces an isomorphism:
		\[\grNvee{N}{r} \simeq \Hom(\grN{r}{N}, R^\times).\]
	\end{lem}

	\begin{proof}
		 The induced map is injective because, for the trace pairing $\tr\colon \got{l}^{-N}\got{d}^{-1} \otimes \got{l}^{-r}\got{o}_K \to \Q_\ell$, $\tr(a\otimes b) \in \Z_\ell$ for all $a\in \got{l}^{-r}$ if and only if $\pi^{-r}b\in \got{d}^{-1}$, if and only if $b\in \got{l}^{r}\got{d}^{-1}$. After the injectivity is established, the proof can then be completed using a general but elementary fact: if $G$ is a finite abelian $\ell$-group killed by $\ell^t$ and $R$ a ring containing a primitive $\ell^t$-th root of unity, then $\#G = \#\Hom(G,R^\times)$. Indeed, as the finite abelian group $\grNvee{N}{r}$ is isomorphic to $\grN{r}{N}$ and is killed by $\ell^{N+r}$, the forementioned fact applies, and shows that $\#(\grNvee{N}{r}) = \#\Hom(\grN{r}{N},R^\times)$. Hence, we see that the map in question, known to be injective, is also surjective by counting.
	\end{proof}
	
	\textbf{Example.}
		$K = \Q_\ell$. It is well-known that $(\underline{\Z/\ell^t})^\vee$ is isomorphic to $\boldsymbol{\mu}_{\ell^{t}}$, the group scheme representing the functor $S\mapsto \{\xi\in S: \xi^{\ell^t} = 1\}$. So, if we choose a primitive $\ell^t$-th root of unity $\zeta_t$, Theorem \ref{thm.etalization} in particular provides an isomorphism $(\Z/\ell^{t})^\vee \simeq \boldsymbol{\mu}_{\ell^{t}} \simeq \underline{\ell^{-t}\Z/\Z} \simeq \underline{\Z/\ell^{t}}$. Furthermore, we may write it down explicitly in terms of affine algebras:
		\[\cali{A}_{\bmu_{\ell^{t}}}\simeq R[X]/(X^{\ell^{t}}-1) \xrightarrow{\sim} \cali{A}_{\underline{\Z/\ell^{t}}} \simeq 
		\oplus_{i\in \Z/\ell^{t}} R\delta_i,\quad X\mapsto \sum_{i\in\Z/\ell^{t}} \zeta_{t}^i \delta_i.\]
		
		This map is an isomorphism, since there exists $f_i(X) = \frac{\prod_{j\ne i}(X-\zeta_t^j)}{\prod_{j\ne i}(\zeta_t^i-\zeta_t^j)} \in R[X]$ (the denominator is allowed as $\ell^{-1} \in R$), which enjoys the property $f_i(\zeta_t^j) = 0$ (resp.~$=1$) if $j\ne i$ (resp.~$j=i$). Moreover, we see here that the isomorphism indeed depends on the $\zeta_{t}$ chosen.
	
	\vspace{3mm}
		
	As $\op$ contains all $\ell$-power roots of unity, all groups $\G_{N,r}$ can be made constant group schemes \textit{by choosing a primitive $\ell^{N+r}$-th root of unity $\zeta_{N,r}$ for each pair $(N,r)$}. So, we have an isomorphism that may or may not be canonical: $\cali{D}(K,\op) \approx \plim_{n}\ilim_{r}\plim_{N}\cali{A}_{\grNvee{N}{r}}\otimes\opn$, where the limit over $n$ is just projection by congruences, and the limit/colimit over $r,N$ are given by, after base change to $\op/p^n$, maps $\psi^{r,N},\psi_{N,r}$ coming from commutative diagrams:
	\[\begin{tikzcd}
	\Hom_{\op}(\cali{C}(\grN{r}{N},\op),\op) \ar[r,"\sim"] \ar[d]&
	\cali{A}_{\G_{N,r}} \ar[r,"\sim"] \ar[d] & \cali{A}_{\GrNvee{N}{r}} \ar[d,"\psi^{r,N}"]\\
	\Hom_{\op}(\cali{C}(\grN{r-1}{N},\op),\op) \ar[r,"\sim"] & 
	\cali{A}_{\G_{N,r+1}} \ar[r,"\sim"] & \cali{A}_{\GrNvee{N}{r+1}}
	\end{tikzcd}\]
	
	and
	\[\begin{tikzcd}
	\Hom_{\op}(\cali{C}(\grN{r}{N+1},\op),\op) \ar[r,"\sim"] \ar[d]&
	\cali{A}_{\G_{N+1,r}} \ar[r,"\sim"] \ar[d] & \cali{A}_{\GrNvee{N-1}{r}} \ar[d,"\psi_{N,r}"]\\
	\Hom_{\op}(\cali{C}(\grN{r}{N},\op),\op) \ar[r,"\sim"] & 
	\cali{A}_{\G_{N,r}} \ar[r,"\sim"] & \cali{A}_{\GrNvee{N}{r}}
	\end{tikzcd}\]

	where all group schemes are over $\op$, and the left vertical maps are induced from restriction and inclusion of functions respectively. We also remind the reader that the two left squares are canonical while the two right squares, in particular $\psi^{r,N}$'s and $\psi_{N,r}$'s, depend on the choice of $\zeta_{N,r}$'s. We would like to identify $\cali{D}(K,\op)$ with uniformly continuous functions by the linearization \eqref{equation.linearization}, so we need a result about ``functoriality'' stated below. Note in the following we do not distinguish $\cali{A}_{\GrNvee{N}{r}}$ from $\cali{C}(\grNvee{N}{r}, \op)$.
	
	\begin{prop}
		\label{prop.functoriality}
		Suppose the roots of unity $\{\zeta_{N,r}\}_{N,r}$ are chosen in such a way that there exists $\{\zeta_n\}_{n\ge 0}$, where each $\zeta_n$ is a primitive $\ell^n$-th root of unity, satisfying $\zeta_{t+1}^\ell = \zeta_t$ and $\zeta_{N,r} = \zeta_{N+r}$. Then the map $\psi^{r,N}$ is induced from the natural projection $\grNvee{N}{r+1} \to \grNvee{N}{r}$, so $\psi^{r,N}$ is inclusion of continuous functions on $\grNvee{N}{r}$ to those on $\grNvee{N}{r+1}$; and the map $\psi_{N,r}$ is induced from the inclusion of groups $\grNvee{N}{r} \to \grNvee{N-1}{r}$, so $\psi_{N,r}$ is restriction of continuous functions on $\grNvee{N-1}{r}$ to those on $\grNvee{N}{r}$.
	\end{prop}
	
	\textit{Remark.} Essentially this proposition says that, when the roots of unity $\zeta_{N,r}$'s are chosen compatibly, the limit/colimit of affine algebras is taken with respect to ``natural'' maps. We note that this ``functoriality'' result is in fact clear in view of $\got{l}$-divisible groups, although we will not need this perspective. For instance, we explain briefly how it works for $\psi_{N,r}$: Fix $r\in \Z_{\ge 0}$ and consider the two $\got{l}$-divsible groups over $\op$, $\Scr{G}_1 = \ilim_{N \to \infty}\GrNvee{N}{r}$ and $\Scr{G}_2 = \ilim_{N \to \infty}(\GrN{r}{N})^\vee = \ilim_{N \to \infty} \G_{N,r}$. To prove the functoriality of $\psi_{N,r}$ it suffices to show that a choice of compatible system of primitive $\ell$-power roots of unity gives an isomorphism between $\Scr{G}_1$ and $\Scr{G}_2$, and this isomorphism restricts to the same isomorphism $\GrNvee{N}{r} \to \G_{N,r}$ as in Theorem \ref{thm.etalization} for all $N$. To give this isomorphism, as both $\Scr{G}_1$ and $\Scr{G}_2$ are \'etale, we can instead give it for the corresponding Tate-modules, $T_{\got{l}}(\Scr{G}_1) = \got{l}^r\got{d}^{-1}$ and $T_{\got{l}}(\Scr{G}_2) = \Hom_{\Z_\ell}(\got{l}^{-r}\got{o}_K,\Z_\ell(1))$. Denote by $\zeta \in \Z_\ell(1)$ the chosen compatible system of $\ell$-power roots of unity, and we can form the following isomorphism of finite free $\got{o}_K$-modules with trivial Galois action:
	\[\got{l}^r\got{d}^{-1} \xrightarrow{\tr} \Hom_{\Z_\ell}(\got{l}^{-r}\got{o}_K,\Z_\ell)\xrightarrow{x \mapsto x\otimes \zeta} \Hom_{\Z_\ell}(\got{l}^{-r}\got{o}_K,\Z_\ell)\otimes_{\Z_\ell} \Z_\ell(1) \simeq \Hom_{\Z_\ell}(\got{l}^{-r}\got{o}_K,\Z_\ell(1)).\] 
	
	Moreover, it is clear that the above map sends $x\in \got{l}^{r}\got{d}^{-1}$ to the linear map whose value at $y\in \got{l}^{-r}\got{o}_K$ is $\zeta^{\tr(xy)}$. In turn, by chasing the definition of the isomorphism in Theorem \ref{thm.etalization}, it can be shown that the two isomorphisms coincide at the piece $\GrNvee{N}{r} \to \G_{N,r}$. 
	
	\begin{proof}[Proof of Proposition \ref{prop.functoriality}]
		We prove here the statement about $\psi^{r,N}$; for $\psi_{N,r}$ it can be proved in a similar fashion. Recall from the proof of Theorem \ref{thm.etalization} that we constructed the isomorphism $\GrNvee{N}{r} \xrightarrow{\sim} \G_{N,r}$ by sheafifying the group homomorphisms $\phi_{r}: \grNvee{N}{r} \to \Hom(\grN{r}{N},R^\times)$. Thus, to prove the statement about group schemes, by functoriality, it suffices to prove the statement for groups. That is, to show the following diagram commutes:
		\begin{displaymath}
		\begin{tikzcd}
		\grNvee{N}{r+1} \ar[d,"\rm proj."] \ar[r,"\phi_{r+1}"] & \Hom(\grN{r+1}{N}, R^\times) \ar[d,"\rm incl.^*"]\\
		\grNvee{N}{r} \ar[r,"\phi_{r}"] & \Hom(\grN{r}{N},R^\times)
		\end{tikzcd}
		\end{displaymath}
	
		where the right vertical arrow is pullback by inclusion of groups. It is easily seen that the commutativity boils down to checking for all $a\in \grNvee{N}{r+1}, b\in \grN{r}{N}$, $\zeta_{N,r+1}^{[\ell^{N+r+1}]\tr(ab)} = \zeta_{N,r}^{[\ell^{N+r}]\tr(ab)}$. As $\tr(ab) \in \ell^{-N-r}\Z_\ell/\Z_\ell$, we see that
		\[
			\zeta_{N,r+1}^{[\ell^{N+r+1}]\tr(ab)} = \zeta_{N,r+1}^{\ell [\ell^{N+r}]\tr(ab)} = (\zeta_{N,r+1}^\ell)^{[\ell^{N+r}]\tr(ab)} = \zeta_{N,r}^{[\ell^{N+r}]\tr(ab)},
		\] 
		
		where the last equality follows from the compatibility of our particular choice of roots of unity.
	\end{proof}
	
	\textbf{Notation.} 
	\begin{enumerate}
		\item When a compatible system of $\ell$-power roots of unity is chosen, say $\{\zeta_n\}_{n\ge 0}$, for any $r \in \Q_\ell$ we shall use the shorthand $\zeta^r$ to mean $\zeta_t^{r\ell^t}$ if $\ord_\ell(r) = -t<0$ and $\zeta^r =1$ if $r\in \Z_\ell$. Of course, in the former case, for any $N\ge t$, $\zeta_N^{r\ell^N} = \zeta_t^{r\ell^t}$, so we really have $\zeta^r = \lim_{N \to \infty} \zeta_N^{r\ell^N}$. Note also this notation is already present in Theorem \ref{thm.etalization} and Lemma \ref{lem.pairing-perfect}, and our use is consistent. Finally we point out for each $r\in \Q_\ell$, the function $\zeta^{rx}: \Q_\ell \to \op$ is uniformly continuous.
		
		\item Instead of saying ``let $\{\zeta_n\}_{n\ge0}$ be a system of primitive $\ell$-power roots of unity such that $\zeta_{t+1}^\ell =\zeta_t$ for all $t$'', we will simply say ``let $\zeta\in T_\ell(\G_m)$ be \textit{generative}'', where \textit{generative} should be understood as either being a topological group generator, or a generator for the $\Z_\ell$-module.
	\end{enumerate}
	 
	We may now summarize our results below; even better, as a byproduct, we recover the familiar formula for the Fourier transform (we are using $\int_K f(x) \mu(x)$ to mean the evaluation of $\mu$ at $f$):
	\begin{cor}
		\label{cor.fourier1}
		Let $\zeta \in T_\ell(\G_m)$ be generative. Then there is an isomorphism depending on $\zeta$:
		\begin{align}
		\label{equation.measure-linearization}
		\cali{D}(K,\op) \xrightarrow{\sim} \plim_{n}\ilim_{r:\rm incl.}\plim_{N: \rm proj.}\cali{C}(\grNvee{N}{r},\opn)
		\hookrightarrow
		\cali{C}(K,\op),
		\end{align}
	
		which identifies $\cali{D}(K,\op)$ as uniformly continuous functions on $K$. Explicitly, it is given by the map:
		\[\mu\longmapsto \hat{\mu}_{\zeta}(r) = \int_{K}\zeta^{\tr(rx)} \mu(x)\qquad\text{for all } r\in K.\]
	\end{cor}

	\textit{Remarks.} 
	\begin{enumerate}
		\item Note that $\zeta^{ab}$ is in general not equal to $(\zeta^a)^{b}$ since $a,b$ take values in $\Q_\ell$.
		
		\item We could have taken $\int_{K}\zeta^{-\tr(rx)} \mu(x)$ to be the definition of $\hat{\mu}_\zeta(r)$ and still get an isomorphism, which is the common convention in classical Fourier analysis on real numbers.
		
		\item When $\zeta$ is fixed, we shall abuse the notation and just write $\hat{\mu}$.
	\end{enumerate}

	\begin{proof}[Proof of Corollary \ref{cor.fourier1}]
		Only the explicit formula requires a proof. We note that 
		\[\cali{D}(K,\op) \simeq \plim_{n}\ilim_{r}\plim_{N}\cali{C}(\grNvee{N}{r},\opn)\]
		
		is obtained by taking the projective limit of
		\[\ilim_{r}\plim_{N} \Hom_{\op}(\cali{C}(\grN{r}{N},\opn),\opn)\simeq \ilim_{r}\plim_{N}\cali{C}(\grNvee{N}{r},\opn),\]
		
		which in turn is obtained by gradually enlarging the isomorphism
		\[\plim_{N} \Hom_{\op}(\cali{C}(\grN{r}{N},\opn),\opn) \simeq \plim_{N}\cali{C}(\grNvee{N}{r},\opn),\]
		
		and which in turn is the projective limit of
		\[\Hom_{\op}(\cali{C}(\grN{r}{N},\opn),\opn) \simeq \cali{C}(\grNvee{N}{r},\opn).\]
		
		The last isomorphism, however, is implicit in steps 2 and 4 in the proof of Theorem \ref{thm.etalization}. Using the notation from there, we see that the map is given by (base change to $\opn$ of)
		\[l \mapsto \sum_{b\in\grNvee{N}{r}}\delta_b l(x_b) = \sum_{b\in\grNvee{N}{r}}\delta_b\sum_{a\in\grN{r}{N}}\zeta^{\tr(ab)}l(a+\got{l}^N)\in \cali{A}_{\grNvee{N}{r}}\simeq \cali{C}(\grNvee{N}{r},\op).\]
		
		Thus, for any $b\in \grNvee{N}{r}$, $\hat{l}(b) = \int_{\grN{r}{N}}\zeta^{\tr(bx)}l(x)$. Chasing back, we then obtain the formula claimed above.
	\end{proof}
	
	\paragraph{Example: Fourier transform on $\Z_\ell$.}
	In the following we fix some $\zeta\in T_\ell(\G_m)$ that is generative. For the $\cali{C}(\Q_\ell,\op)$-module $\cali{D}(\Q_\ell,\op)$, we have a submodule $\cali{D}(\Z_\ell,\op) = \cali{D}(\Q_\ell,\op)\1_{\Z_\ell}$, which has a compatible linearization:
	\[
	\begin{tikzcd}
		\cali{D}(\Q_\ell,\op) \ar[r,"\sim"] &\plim_{n}\ilim_{r}\plim_{N}
		\cali{C}(\ell^{-N}\Z_\ell/\ell^r\Z_\ell,\opn) \ar[r,hook]
		&
		\cali{C}(\Q_\ell,\op)\\
		\cali{D}(\Z_\ell,\op) \ar[r,"\sim"] \ar[u,hook]
		&\plim_{n}\plim_{N} \cali{C}(\ell^{-N}\Z_\ell/\Z_\ell,\opn) \ar[r,equal] \ar[u,hook]
		&\cali{C}(\Q_\ell/\Z_\ell,\op) \ar[u,hook]
	\end{tikzcd}
	\]
	
	where the last vertical arrow is given by regarding a function on $\Q_\ell/\Z_\ell$ as a continuous function on $\Q_\ell$ that is constant on $\Z_\ell$-cosets. Note that since $\Q_\ell/\Z_\ell$ is discrete, $\cali{C}(\Q_\ell/\Z_\ell,\op)$ really consists of arbitrary functions from $\Q_\ell/\Z_\ell$ to $\op$. Moreover, the isomorphism $\cali{D}(\Z_\ell,\op) \xrightarrow{\sim} \cali{C}(\Q_\ell/\Z_\ell,\op)$ recovers the Fourier transform utilized in \cite{Si87}, especially Proposition 2.1 \textit{loc.~cit.}, by our choice of a generative $\zeta$. To illustrate this isomorphism, below we list some elements in $\cali{C}(\Q_\ell/\Z_\ell,\op)$ and the corresponding elements in $\cali{D}(\Z_\ell,\op)$:
	\begin{itemize}
		\item The Dirac function $\underline{\delta_a}$, where $a\in \Q_\ell/\Z_\ell$, is the Fourier transform of $\nu_a(x) =\zeta^{-ax}dx$, where $dx$ is the Haar measure on $\Z_\ell$ with $\int_{a+\ell^n\Z_\ell}dx = \ell^{-n}$ for all $a\in \Q_\ell,n\in \Z_{\ge 0}$. In particular, we have $\widehat{dx} = \underline{\delta_0}$.
		
		\item The function $\zeta^{bx}$ for any $b\in \Z_\ell$ is the Fourier transform of the Dirac measure $\overline{\delta_b}$. Note that if $b\notin \Z_\ell$, then $\zeta^{bx}$ will not be constant on $\Z_\ell$-cosets, meaning that in this case $\zeta^{bx}$ is not in the image of $\cali{D}(\Z_\ell,\op)$.
		
		\item Let $\chi: (\Z_\ell/\ell)^\times\to \C_p^\times$ be a non-trivial character. Let $R_\chi$ be the function on $\Q_\ell/\Z_\ell$ given by (\textit{cf.~}§4 of \cite{Si87})
		\begin{align*}
			R_\chi(x) = 
			\begin{cases}
				0 & \text{if } x\in \frac{1}{\ell}\Z_\ell/\Z_\ell;\\
				\sum_{1\le a<\ell} \chi(a)\frac{\zeta^{ax}}{1-\zeta^{\ell x}} & \text{otherwise.}
			\end{cases}
		\end{align*}
		
		Via direct computation, one can show that $R_\chi = \hat{\mu}_\chi$, where $\mu_\chi\in \cali{D}(\Z_\ell,\op)$ is such that for any $n\in \Z_{\ge 0}$ and $0\le a<\ell^n$:
		\begin{align*}
			\mu_\chi(a+\ell^n\Z_\ell) = \begin{cases}
				0 & \text{if } n=0,1;\\
				\chi(a)\ell^{1-n}(V(a) - \frac{\ell^{n-1}-1}{2}) & \text{if }n\ge 2,
			\end{cases}
		\end{align*}
		
		where for $x = \sum_{i\ge 0} x_i \ell^i \in \Z_\ell$ with $0\le x_i<\ell$ for all $i$, $V(x) = \sum_{i\ge 0} x_{i+1}\ell^i$. As such, in the terminology of \cite{Si87}, a rational function measure attached to $\sum_{1\le a<\ell}\chi(a)\frac{X^a}{1-X^\ell}$ necessarily takes the form $\mu_\chi + \sum_{a\in \Q_\ell/\Z_\ell} \lambda_a \nu_a$, where all but finitely many $\lambda_a\in \C_p$ are zero. 
	\end{itemize}
	
	\subsection{The Fourier tranform}
	
	Recall our setting: $K/\Q_\ell$ is finite, $\got{o}_K$ is the valuation ring, $\got{l}$ is the maximal ideal, and $\pi$ is a chosen uniformizer. Let further $q= \nm(\got{l})$, where $\nm(\got{a}) = \#\got{o}_K/\got{a}$ for any nonzero ideal $\got{a}\subset \got{o}_K$, $\got{d}$ be the different ideal of $\got{o}_K/\Z_\ell$, and $\delta \in \Z_{\ge 0}$ be such that $q^{\delta} = \nm(\got{d})$. 
	
	\vspace{3mm}
	
	Although the isomorphism in Corollary \ref{cor.fourier1} already deserves to be called the Fourier transform (attached to some generative $\zeta$), for our purpose we need a finer version that conforms to the classical (e.g., $L^2$-) Fourier theory. Namely we want to introduce Haar measures and obtain an endomorphism of functions on $K$. Following Tate, we denote the Haar measure of mass $\nm{\got{d}}^{-1/2}=q^{-\delta/2}$ on $\got{o}_K$ by $dx$. Recall, a priori, it is a linear functional on the space of $\got{o}_p$-valued Schwartz-Bruhat functions on $K$:
	\[
		\cali{S} = \cali{S}(K,\got{o}_p) = \left\{\sum_{\substack{E\subset K:\\\textrm{open compact}}}
		\lambda_{E}\1_{E} : \lambda_{E}\in \op, \textrm{ almost all }\lambda_{E}=0\right\} = \ilim_{N,r}\cali{C}(\grN{r}{N},\op),
	\]
	
	and it pairs with $\1_{a+\got{l}^N}$ to $q^{-N-\delta/2}$. We remark that $dx\in \Hom_{\op}(\cali{S},\op)$ is not in $\cali{D}(K, \got{o}_p)$, which follows from
	\begin{lem}
		\label{lem.S-extension}
		Let $L \in \cali{D}(K,\op)$ and recall the canonical projection $p_{n,r}: \cali{C}(K,\got{o}_p) \to 
		\cali{C}(\got{l}^{-r}\got{o}_K,\got{o}_p/p^n)$ in Proposition \ref{prop.topological-linearization}. With the linearization in \eqref{equation.pre-linearization-measure}, modulo $p^n$, $L\in \plim_{N}\Hom_{\op}(\cali{C}(\grN{r}{N},\op),\opn)$ if and only if $L$ is zero when restricted to $\ker(p_{n,r})$. Conversely, for $l \in \Hom_{\op}(\cali{S},\op)$ there is at most one $L \in \cali{D}(K,\op)$ such that $L|_{\cali{S}} = l$; and the existence is guaranteed if: 
		\begin{align}
		\tag{*}\label{condition.star}
		\text{for all } n\ge 1, \text{ there exists } r \ge 0, \text{ such that } l|_{\cali{S}\cap \ker(p_{n,r})} = 0.
		\end{align}
	\end{lem}
	
	\textit{Remark.} Clearly, \textit{restriction to some $\ker(p_{n,r})$ is zero} means the linear functional is compactly supported modulo $p^n$. Since for any fixed $n>0$, $dx(\1_{\pi^{-r}+\got{o}_K}) = q^{-\delta/2} \not\equiv 0 \bmod p^n$ for all $r>0$, we see that $dx$ is not the restriction of any $L\in \cali{D}(K,\op)$ to $\cali{S}$.
	
	\begin{proof}[Proof of Lemma \ref{lem.S-extension}]
		By the proof of Corollary \ref{cor.measure-linearization}, the first claim is clear. Next consider any $l \in \Hom_{\op}(\cali{S},\op)$: since for any $n,r$, $\cali{C}(K,\op)/\ker(p_{n,r}) = \cali{C}(\got{l}^{-r}\got{o}_K,\opn) = \cali{S}/\cali{S}\cap \ker(p_{n,r})$, we see that $\cali{S}$ is dense in $\cali{C}(K,\op)$. Thus there is at most one $L$ extending $l$. Now for any $l$ satisfying \eqref{condition.star}, for each $n\ge 1$, $l$ naturally gives rise to $\l_n \in \Hom_{\op}(\cali{S}/\cali{S}\cap\ker(p_{n,r}),\opn) = \Hom_{\op}(\cali{C}(K,\op)/\ker(p_{n,r}),\opn)\subset \ilim_{r}\Hom_{\op}(\cali{C}(\got{l}^{-r}\got{o}_K,\op),\opn)$. When we vary $n$, the $l_n$'s are apparently compatible and hence form 
		\[L = (l_n)_n \in \plim_{n}\ilim_{r}\Hom_{\op}(\cali{C}(\got{l}^{-r}\got{o}_K,\op),\opn) \simeq \cali{D}(K,\op).\] 
		
		It is clear that $L$ extends $l$.
	\end{proof}
	
	Although $dx$ is not in $\cali{D}(K,\op)$, we have the following easy consequence of Lemma \ref{lem.S-extension}:
	\begin{cor}
		\label{cor.basic-functions}
		We have $\1_{\got{o}_K}(x)dx\in \cali{D}(K,\got{o}_p)$. In fact, given a generative $\zeta\in T_\ell(\G_m)$, $(\1_{\got{o}_K}dx)^\wedge_{\zeta} = q^{-\delta/2}\1_{\got{d}^{-1}}$. Similarly, we have $\1_{a+\got{l}^N}dx \in \cali{D}(K,\op)$, and $(\1_{a+\got{l}^N} dx)^\wedge_{\zeta}(y) = \zeta^{\tr(ay)}q^{-N-\delta/2}\1_{\got{l}^{-N}\got{d}^{-1}}$.
	\end{cor}
	
	\begin{proof}
		Since $\1_{\got{o}_K}(x)dx$ is compactly supported, by Lemma \ref{lem.S-extension}, it is in $\cali{D}(K,\op)$. Now for $r\in \pi^{k-\delta}\got{o}_K^\times$, we compute: 
		\begin{displaymath}
			(\1_{\got{o}_K}dx)^\wedge_{\zeta}(r) = \int_{\got{o}_K} \zeta^{\tr(rx)}dx 
			= 
			\begin{cases*}
				\nm(\got{d})^{-1/2} & $k\ge 0$;\\
				\nm(\got{d})^{-1/2}\sum_{a\in\got{o}_K/\got{l}^{-k}}
				\zeta^{\tr(ar)} = 0 & $k< 0$.
			\end{cases*}
		\end{displaymath}
	
		The claim for $\1_{a+\got{l}^N}$ then follows from standard properties of the Fourier transform, e.g., using i) and ii) of Proposition \ref{prop.classic}.
	\end{proof}
	
	\begin{de}
		A \textit{dissipative} function in $\cali{C}(K,\got{o}_p)$ is one that is supported on a compact subset of $K$ modulo $p^n$ for each $n\in \Z_{>0}$.
	\end{de}
	
	\textit{Remarks.} 
	\begin{enumerate}
		\item For a continuous function $f$ on $K$, the condition that $f$ is compactly supported modulo $p^n$ for all $n\ge 0$ will often by paraphrased as $\lim_{x\to \infty} f(x) = 0$.
		
		\item For simplicity, we have included in the definition that dissipative functions are valued in $\op$. However, everything to follow can be extended to ``$\C_p$-valued dissipative functions'', namely those of the form $f = p^{-n} g$ for some $n\in \Z_{\ge 0}$ and dissipative $g\in \cali{C}(K,\op)$.
	\end{enumerate}
	
	\vspace{3mm}
	
	As an example, the characteristic function $\1_{a+\got{l}^N}$ is dissipative since it is already compactly supported. Clearly, dissipative functions form a module over $\cali{C}(K,\op)$. We also have a linearization of dissipative functions compatible with Proposition \ref{prop.topological-linearization} by $\plim_{n}\ilim_{r}\ilim_{N:\textrm{\rm{incl.}}} \cali{C}(\grN{r}{N},\opn)$, where the colimit with respect to $r$ is extension by zero. As such, for all dissipative $f$ and for all $n\ge 0$, there exists $f_n\in \cali{S}$ such that $\sup_{x\in K}|f(x)-f_n(x)|_p <p^{-n}$. The notion of dissipative functions is tailored for $dx$, because again $f(x)dx$ modulo $p^n$ is compactly supported for all $n$. Thus using Lemma \ref{lem.S-extension} again, we conclude:
	
	\begin{cor}
		\label{cor.haar-dissipative-matching}
		For any dissipative $f$, $f(x)dx \in \cali{D}(K,\op)$.
	\end{cor}
	
	\begin{de}
		For any generative $\zeta\in T_\ell(\G_m)$, we define the \textit{Fourier transform} of $f$ with respect to $\zeta$ by $(fdx)^\wedge_{\zeta}$. In the rest of this paper, we shall fix such a $\zeta$ and simply write $\hat{f}$ for $(fdx)^\wedge_{\zeta}$.
	\end{de}
	
	\textit{Remark.} Note that our formal definition of the Fourier transform coincides with the contrived one (denoted by $\cali{F}$) used in the appendix. Indeed, the only ingredient for that definition is a field embedding $\iota_p: \Q(\mu_{\ell^\infty}) \hookrightarrow \C_p^\times$, which is the same as choosing a generative $\zeta = \{\zeta_n\}_{n\ge 0} \in T_\ell(\G_m)$ through the correspondence
	\[\iota_p(e^{-2\pi i/\ell^n}) = \zeta_n.\]
	
	In turn, it is easy to verify that $\iota_p(e^{-2\pi it}) = \zeta^{t}$ for all $t \in \Q_\ell$, and hence by Corollary \ref{cor.basic-functions} we have $\cali{F}(\1_{a+\ell^N\Z_\ell}) = (\1_{a+\ell^N\Z_\ell})^\wedge_\zeta$ for all $a\in \Q_\ell$ and $N\in \Z$. This is enough to conclude the coincidence by approximation.
	
	\begin{thm}[Fourier inversion]
		\label{thm.fourier-inversion}
		The Fourier transform of a dissipative $f$ is still dissipative. Moreover,
		\begin{align*}
			((fdx)^\wedge_\zeta dx)^\wedge_{\zeta^{-1}} = f.
		\end{align*}
	\end{thm}
	
	\textit{Remark.} A dissipative function is uniformly continuous, as a compactly supported function is. A \textit{tour de force} proof of this fact is to combine Theorem \ref{thm.fourier-inversion} and Corollary \ref{cor.fourier1}.
	
	\begin{proof}[Proof of Theorem \ref{thm.fourier-inversion}]
		First we want to show, modulo $p^n$ for any $n\ge 1$, $\hat{f}$ is compactly supported on $K$. Note that if $g$ is dissipative then $\hat{g}\in\cali{C}(K,\op)$ and $(p^n g)^\wedge = p^n\hat{g}\in p^n\cali{C}(K,\op)$. Now, if we take $f_n\in \cali{S}$ such that $\sup_{x\in K}|f(x) - f_n(x)|_p <p^{-n}$, then $f_n-f = p^n g$ where $g$ is continuous, and is in fact dissipative because $\lim_{x\to \infty} g(x) = p^{-n}\lim_{x\to \infty}[f(x)-f_n(x)] = 0$. This implies $\hat{f_n}-\hat{f}\in p^n\cali{C}(K,\op)$, so $\hat{f_n}\equiv \hat{f}\bmod p^n$. To complete our argument, it suffices to note that $f_n$ may be regarded as some function on $\grN{r}{N}$ for some $r,N$, and thus $\hat{f_n}$ is naturally a function on $\grNvee{N}{r}$ by Cartier duality, and thus compactly supported. As such, $\hat{f}$ is compactly supported modulo $p^n$.
		
		\vspace{3mm}
		
		The inversion identity can be directly verified for characteristic functions of open compact subsets, and hence it holds true for all elements of $\cali{S}$. Hence it holds true for dissipative function by \textit{approximation}: just like above, for any dissipative $f$, we take $f_n$ and we see $(fdx)^\wedge_\zeta = (f_ndx)^\wedge_\zeta + p^n(gdx)^\wedge_\zeta$ and thus $((fdx)^\wedge_\zeta dx)^\wedge_{\zeta^{-1}} 
		= f_n + p^n((gdx)^\wedge_\zeta dx)^\wedge_{\zeta^{-1}}$, whereby $((fdx)^\wedge_\zeta dx)^\wedge_{\zeta^{-1}} = \lim_n f_n = f$.
	\end{proof}

	\begin{prop}
		\label{prop.classic}
		Let $f,g$ be dissipative. We have the following facts:
		\begin{enumerate}
			\item[i)] If $g(x) = f(x-h)$, then $\hat{g}(y) = \hat{f}(y)\zeta^{\tr(hy)}$;
			
			\item[ii)] If $g(x) = f(\lambda x)$, then $\hat{g}(y)=\chihaar(\lambda)^{-1}\hat{f}(y/\lambda)$, where $\chihaar(r)dx = d(rx)$;
			
			\item[iii)]\emph{(Ultrametric inequality)}
			$\left|\int_{K}f(x)dx\right|_p\le \sup_{x\in K}|f(x)|_p$.
			
			\item[iv)]
			\emph{(Poisson summation formula)}
			For any $t\in K$:
			\begin{align*}
				\int_{\got{o}_K}f(x+t)dx = q^{-\delta/2}\int_{\got{d}^{-1}} \hat{f}(y)\zeta^{-\tr(ty)}dy.
			\end{align*}
		
			\item[v)] \emph{(Riemann-Lebesgue)} $\lim_{t \to \infty} \int_K f(x)\zeta^{\tr(tx)}dx = 0$. Furthermore, if $\int_K f(x) dx = 0$ then $\lim_{t\to  0} \int_K f(x)\zeta^{\tr(tx)}dx = 0$.
		\end{enumerate}
	\end{prop}

	\begin{proof}
		Assertions i)-iv) are easily checked for $\cali{S}$, where the general case follows by approximation. We now explain v): Firstly note that both limits are limits of the value $\hat{f}(t)$. Next since $f$ is dissipative, by Theorem \ref{thm.fourier-inversion}, $\hat{f}$ is also dissipative. Therefore $\lim_{t \to \infty}\hat{f}(t)=0$, which is the first limit. For the second limit, we note $\hat{f}(0) = \int_{K} f(x)dx = 0$, and since $\hat{f}$ is continuous at 0, we deduce that $\lim_{t\to 0}\hat{f}(t)=0$.
	\end{proof}
	
	\begin{cor}
		\label{cor.uniform-convergence}
		Suppose a sequence of dissipative functions $(f_n)_{n\ge 0}$ converges to $f$ uniformly, i.e., $\lim_{n\to \infty}\sup_{x\in K}|f(x)-f_n(x)| = 0$. Then $\lim_{n\to \infty}\int_K f_n(x) dx = \int_{K} f(x)dx$.
	\end{cor}
	
	\begin{proof}
		By the ultrametric inequality from Proposition \ref{prop.classic}, we have
		\[\left|\int_K f(x)dx - \int_K f_n(x)dx\right|_p \le \sup_{x\in K}|f(x) - f_n(x)|_p,\]
		
		which tends to zero when $n$ tends to infinity by definition.
	\end{proof}

	\section{Application: local functional equation}
	
	In what follows, by a multiplicative character, or a character for short, we exclusively mean a continuous group homomorphism $\chi: K^\times \to \C_p^\times$. Clearly such a $\chi$ is uniquely determined by $\chi(\pi) \in \C_p^\times$ and its restriction $\chi|_{\got{o}_K^\times}$, which is of finite image. We shall call $|\chi(\pi)|_p \in \R_+$ the \textit{modulus} of $\chi$ and denote it by $\norm{\chi}$, which is our $p$-adic surrogate of the exponent of a complex character. By convention, when $\chi|_{\got{o}_K^\times}$ is trivial, the character $\chi$ is called \textit{unramified}, and we denote by $\chi_\lambda$ the unramified character with $\chi(\pi) = \lambda$. Following \cite{Ta50}, we define an equivalence relation between multiplicative characters, by announcing $\chi \sim \psi$ if $\chi\psi^{-1} = \chi_\lambda$ for some $\lambda \in \C_p^\times$. We note the following simple fact:
	
	\begin{prop}
		\label{prop.rigid-parametrization}
		An equivalence class $\got{C}$ of characters can be parametrized by the rigid analytic space $\C_p^\times$, which formally is the $\C_p$-points of the rigid analytification of $\G_m$ over $\C_p$. The parametrization that we will employ throughout, is to let any character $\chi\in \got{C}$ correspond to the point $\chi(\pi) \in \C_p^\times$, and conversely, let $\lambda \in \C_p^\times$ correspond to $\tilde{\chi}\chi_\lambda\in \got{C}$, where $\tilde{\chi}$ is the unique character in the class with $\tilde{\chi}(\pi) = 1$. 
	\end{prop}
	
	\begin{de}
		\label{definition.schwartz-class}
		Let $c$ be a positive real number. A dissipative function $f$ is said to be \textit{of Schwartz class} $c$, if for any multiplicative character $\chi$ of modulus $c$, $\lim_{x\to 0}f(x)\chi(x)$ exists and, with $f\chi(0)$ defined to be $\lim_{x\to 0}f(x)\chi(x)$, the resulting function $f\chi$ on $K$ is dissipative. 
	\end{de}
	
	\begin{prop}
		\label{prop.schwartz-criterion}
		Let $c\ge 1$ and $f$ be a dissipative function. If $f$ is of Schwartz class $c$, then for all multiplicative characters $\chi$ of modulus $c$, $\lim_{x\to 0}f(x)\chi(x)=0$. Conversely, if for some character $\chi$ of modulus $c$, $\lim_{x\to 0}f(x)\chi(x) = 0$, then $f$ is of Schwartz class $c$.
	\end{prop}
	
	\begin{proof}
		Suppose first $f$ is of Schwartz class $c$, and let $\chi$ be a multiplicative character of modulus $c$, and $\lambda\in \op^\times\setminus \{1\}$. By the definition of Schwartz class, we know that both $f\chi$ and $f\chi\chi_\lambda$ are dissipative, and both $\lim_{x\to 0}f(x)\chi(x)$ and $\lim_{x\to 0}f(x)\chi\chi_\lambda(x)$ exist. We assert that $\lim_{x\to 0}f(x)\chi(x)=0$, since otherwise the limit 
		\[
		\lim_{n\to \infty} \lambda^n
		= \lim_{n\to \infty} \chi_\lambda(\pi^n) 
		= \frac{\lim_{n\to \infty}f\chi\chi_\lambda(\pi^n)}{\lim_{n\to \infty}f\chi(\pi^n)}
		\]
		
		would exist, which is not possible as $\lambda \in \op^\times\setminus \{1\}$.
		
		\vspace{3mm}
		
		Conversely, suppose $\lim_{x\to 0}f(x)\chi(x)=0$ for some character $\chi$ of modulus $c$. Recall that a function $g$ is dissipative if and only if $g$ is continuous on $K$ and $\lim_{x\to \infty} g(x) = 0$. So, to show $f$ is of Schwartz class $c$, we will show that, for any character $\chi'$ of modulus $c$, the limit $\lim_{x\to 0}f\chi'(x)$ exists, and the resulting function $f\chi'$ on $K$ is continuous and $\lim_{x\to \infty}f\chi'(x)=0$. Note first that, when $\chi$ and $\chi'$ are of the same modulus, the function $\chi'\chi^{-1}$ on $K^\times$ is bounded, and thus $\lim_{x\to 0}f(x)\chi'(x) = \lim_{x\to 0}[f\chi(x)\chi'\chi^{-1}(x)] = 0$. It is then clear that $f\chi'$ is continuous on $K$, since on $K^\times$ it coincides with the continuous function $f(x)\chi'(x)$, and at $0$, we defined $f\chi'(0) = \lim_{x\to 0}f\chi'(x)$. Finally to show $\lim_{x\to \infty}f\chi'(x) = 0$, we first note that on $K\setminus \got{o}_K$, $\chi'$ is a bounded function since its modulus is $c\ge 1$. Therefore, as $f$ is dissipative, we have  $\lim_{x\to \infty}f(x)=0$, and subsequently $\lim_{x\to \infty}f(x)\chi'(x)=0$.
	\end{proof}
	
	\begin{cor}
		\label{cor.schwartz-filtration}
		Let $c\ge 1$, $f$ be a dissipative function, and $\chi$ be a multiplicative character of modulus $c$ such that $f\chi$ is dissipative. Then for all $c'\in [1,c)$, $f$ is of Schwartz class $c'$. Furthermore, if there is another character $\chi'\ne \chi$ of modulus $c$ such that $f\chi'$ is dissipative, then $f$ is also of Schwartz class $c$.
	\end{cor}
	
	\begin{proof}
		To show that $f$ is of Schwartz class $c'$ with $1\le c'<c$, by Proposition \ref{prop.schwartz-criterion}, it suffices to show that, for all characters $\chi'$ of modulus $c'$, $\lim_{x\to 0}f\chi'(x) = 0$. First note that since $f\chi$ is dissipative, $\lim_{x\to 0}f\chi(x)$ exists. Next, as $c'<c$ we see that for $x\in \pi^n\got{o}_K^\times$, $|\chi'\chi^{-1}(x)|_p=(c'/c)^n$, and thus $\lim_{x\to 0}\chi'\chi^{-1}(x) = 0$. Combining these, we conclude that
		\[\lim_{x\to 0}f\chi'(x) = \lim_{x\to 0}f\chi(x)\cdot \lim_{x\to 0}\chi'\chi^{-1}(x) = 0.\]
		
		Now suppose $\chi'\ne \chi$ is also of modulus $c$ and $f\chi'$ is dissipative. Again by Proposition \ref{prop.schwartz-criterion}, we want to show $\lim_{x\to 0}f\chi(x)=0$, so as to conclude that $f$ is of Schwartz class $c$. When $\chi'(\pi) \ne \chi(\pi)$, assuming by contradiction that $\lim_{x\to 0}f\chi(x)\ne 0$, then the same argument as in the proof of Proposition \ref{prop.schwartz-criterion} shows that $\lim_{n\to \infty}\chi'\chi^{-1}(\pi^n)$ exists, which is not possible since $\chi'\chi^{-1}(\pi)\in \op^\times\setminus\{1\}$. When $\chi'(\pi) = \chi(\pi)$, since $\chi'\ne \chi$, we must have some $t\in \got{o}_K^\times$ such that $\chi'(t) \ne \chi(t)$. As such, if $\lim_{x\to 0}f(x)\chi(x)\ne 0$, then for any $u\in \got{o}_K^\times$:
		
		\[
			\frac{\lim_{x\to 0}f\chi'(x)}{\lim_{x\to 0}f\chi(x)} 
			= \lim_{n\to \infty} \frac{f\chi'(\pi^n u)}{f\chi(\pi^n u)} = \lim_{n\to \infty}\chi'\chi^{-1}(u) = \chi'\chi^{-1}(u).
		\]
		
		However, this cannot be consistent since, by taking $u=t$ and $u = 1+\pi^n$ where $n$ is big enough such that $\chi(1+\pi^n) =\chi'(1+\pi^n)=1$, we get the contradiction $\chi'\chi^{-1}(t) =1$. 
	\end{proof}
	
	\textit{Remark.} In particular, for $1\le c'\le c$, a dissipative function $f$ of Schwartz class $c$ is automatically of Schwartz class $c'$.

	\begin{cor}
		\label{cor.integral expansion}
		Let $\chi$ be of modulus $c\ge 1$, and $f$ be of Schwartz class $c$. Then we have
		\[
		\int_K f\chi(x) dx 
		= \lim_{n\to \infty}\int_{\pi^{-n}\got{o}_K \setminus \pi^n\got{o}_K}f(x)\chi(x)dx
		= \sum_{n\in \Z}\int_{\pi^n\got{o}_K^\times}f(x)\chi(x)dx.
		\]
	\end{cor}
	
	\begin{proof}
		By Proposition \ref{prop.schwartz-criterion} we have $f\chi(0) = 0$, and thus the compactly supported function $\1_{\pi^{-n}\got{o}_K\setminus\pi^n\got{o}_K}(x)f(x)\chi(x)$ converges to $f\chi(x)$ uniformly when $n$ tends to infinity. By Corollary \ref{cor.uniform-convergence} this implies $\int_{\pi^{-n}\got{o}_K \setminus \pi^n\got{o}_K}f\chi(x)dx\to \int_{K}f\chi(x)dx\ (n\to\infty)$. The second equalily then follows from the first.
	\end{proof}
	
	\begin{de}
		Suppose for a dissipative function $f$, there exists some $c\ge 1$ such that $f$ is of Schwartz class $c$. We define the (local) \textit{zeta integral} attached to $f$ to be:
		\[
		Z(f,\chi) = \int_{K}f(x)\chi(x) dx,
		\]
		
		where $\chi$ ranges over multiplicative characters of modulus in $[1,c]$.
	\end{de}
	
	We are now ready to establish an analogue to Lemma 2.4.1 from \cite{Ta50}. 
	
	\begin{lem}
		\label{lem.analyticity}
		Let $f$ be of Schwartz class $c\ge 1$. For any equivalence class $\got{C}$ of multiplicative characters with the rigid analytic parametrization by $\C_p^\times$ as in Proposition \ref{prop.rigid-parametrization}, the zeta integral $Z(f,\chi)$ is a rigid analytic function on the Laurent domain $\{\chi \in \got{C}: 1\le \norm{\chi}\le c\} \subset \got{C}\simeq \C_p^\times$.
	\end{lem}
	
	\begin{proof}
		Recall we denote by $\tilde{\chi}$ the unique character in the equivalence class such that $\tilde{\chi}(\pi) = 1$. By Corollary \ref{cor.schwartz-filtration}, we know the function $Z(\lambda)=Z(f,\tilde{\chi}\chi_\lambda)$ is defined for $1\le |\lambda|_p\le c$. It then suffices to establish the Laurent expansion of $Z(\lambda)$. Choose some $\xi \in \C_p^\times$ such that $|\xi|_p = c$. For any $\lambda$ such that $1\le |\lambda|_p \le c = |\xi|_p$, we have by Corollary \ref{cor.integral expansion}
		\begin{align*}
			\int_{K} f\tilde{\chi}\chi_\lambda(x)dx
			& = \sum_{n\in \Z} \int_{\pi^n\got{o}_K^\times} f\tilde{\chi}\chi_\lambda(x)dx\\
			& = \sum_{n\ge 0} \left[\int_{\pi^n\got{o}_K^\times}f\tilde{\chi}\chi_\xi(x)dx\right]
			\left(\frac{\lambda}{\xi}\right)^n
			+ \sum_{n>0}\left[\int_{\pi^{-n}\got{o}_K^\times} f\tilde{\chi}(x)dx\right]\lambda^{-n}\\
			& = \sum_{n\ge 0} a_n \left(\frac{\lambda}{\xi}\right)^n + \sum_{n>0}a_{-n} \lambda^{-n}.
		\end{align*}
		
		We note that the above expansion is legitimate, because the coefficient $a_n$ tends to zero when $|n|$ tends to infinity: this can be seen by the ultrametric inequality
		\[
			\left|\int_{\pi^n\got{o}_K^\times}f\chi(x)dx\right|_p \le \sup_{x\in \pi^n\got{o}_K^\times}|f\chi(x)|_p,
		\]
		
		for $\chi\in \{\tilde{\chi},\tilde{\chi}\chi_\xi\}$, and the vanishing of $f\tilde{\chi}$ at infinity and $f\tilde{\chi}\chi_\xi$ at zero. In particular the above expansion gives an element in the affinoid algebra $\C_p\chx{X/\xi,Y}/(XY-1)$, which corresponds to the Laurent domain $\{\lambda\in \C_p^\times : 1\le |\lambda|_p\le |\xi|_p = c\}$. 
	\end{proof}
	
	\bk
	
	We can now start to discuss the $p$-adic local functional equation of zeta integrals. For a multiplicative character $\chi$, we write $\chi^*=(\chi\chihaar)^{-1}$, where $\chihaar = \chi_{q^{-1}}$ is the unramified character such that $d(ax) = \chihaar(a)dx$ for all $a\in K^\times$.
	
	\begin{lem}
		\label{lem.local-fe}
		For all multiplicative characters $\chi$ of modulus $1$, and dissipative functions $f,g$ such that all of $f,g,\hat{f},\hat{g}$ are of Schwartz class $1$, we have:
		\begin{align}
			\label{equation.common-zeta-ratio}
			Z(f,\chi)Z(\hat{g},\chi^*)=Z(\hat{f},\chi^*)Z(g,\chi).
		\end{align}
	\end{lem}
	
	\begin{proof}(Tate) 
		We break down the proof in 4 steps:
		
		\vspace{3mm}
		
		\textit{Step 1.} For any function $f(x,y)$ on $K\times K$ that is continuous and compactly supported modulo $p^n$ for all $n\ge 1$, the 2-dimensional integral $\int_{K\times K}f(x,y)dxdy$ exists. In fact, as in the one-dimensional case, we can approximate $f(x,y)$ uniformly by linear combinations of characteristic functions of the form $L=\sum_{m,n}\lambda_{m,n}\1_{E_m\times F_n}$, where the sum is finite and $E_n$ and $F_m$ are open compact. We can then define the double integral to be 
		\[
		\int_{K\times K}f(x,y)dxdy =
		\lim_L \left[\mu_{\rm Haar}(L) = \sum_{m,n}\lambda_{m,n}\mu_{\rm Haar}(E_m)\mu_{\rm Haar}(F_n)\right],
		\] 
		
		where the limit is over all such approximations, and for a compact open subset $E\subset K$, $\mu_{\rm Haar}(E)=\int_E dx$. This limit is well-defined, since if $L,L'$ are two approximations of $f$ such that $\sup_{x\in K}|f(x)-L(x)|_p\le p^{-n}$ and $\sup_{x\in K}|f(x)-L'(x)|_p\le p^{-n}$, then by the ultrametric property we have $|\mu_{\rm Haar}(L)-\mu_{\rm Haar}(L')|_p \le \sup_{x\in K}|L(x)-L'(x)|\le p^{-n}$.
		
		\vspace{3mm}
		
		\textit{Step 2.} Let $f(x,y)$ be supported on $E\times K$, where $E$ is open compact of $K^\times$. We prove the change of variable formula
		\[\int_{K\times K}f(x,y)dxdy = \int_{K\times K}f(x,xy)\chihaar(x)dxdy.\] 
		
		By approximation it suffices to prove this for $f(x,y) = \1_{a+\pi^n\got{o}_K}(x)\1_{b+\pi^m\got{o}_K}(y)$; passing to a finer division if necessary, we may further assume that $a\ne 0$ and $n>\max\{\ord_\pi(a),\ord_\pi(a)+m-\ord_\pi(b)\}$. In this case, the integral on the left is evaluated to be $q^{-n-m-\delta}$. To deal with the integral on the right, put $r = \ord_\pi(a)$ and $a = \pi^{r}u$ with $u\in \got{o}_K^\times$. Given $x\in a+\pi^n\got{o}_K$, we may write $x = a + \pi^nz = \pi^r(u+\pi^{n-r}z)$ for some $z\in \got{o}_K$, and let $z'\in \got{o}_K$ be such that $u^{-1}+\pi^{n-r}z' = (u+\pi^{n-r}z)^{-1}$; this is possible since $n>r=\ord_\pi(a)$. As such, we have
		\begin{align*}
			\1_{b+\pi^m\got{o}_K}(xy) =
			\1_{b+\pi^m\got{o}_K}(\pi^r(u+\pi^{n-r}z)y) = \1_{(u^{-1}+\pi^{n-r}z')(\pi^{-r}b+\pi^{m-r}\got{o}_K)}(y) = \1_{a^{-1}b+\pi^{m-\ord_\pi(a)}\got{o}_K}(y).
		\end{align*}
		
		Using this, we are ready to compute the integral on the right:
		\[
			\int_{K\times K}\1_{a+\pi^n\got{o}_K}(x)\1_{b+\pi^m\got{o}_K}(xy)q^{-\ord_\pi(a)}dxdy = q^{-\ord_\pi(a)}\int_{K\times K}\1_{a+\pi^n\got{o}_K}(x) \1_{a^{-1}b+\pi^{m-\ord_\pi(a)}\got{o}_K}(y) dxdy = q^{-n-m-\delta}.
		\]
		
		So the equality is verified.
		
		\vspace{3mm}
		
		\textit{Step 3.} For dissipative functions $f,g$ such that $f,g,\hat{f},\hat{g}$ are of Schwartz class 1, we show that $\Psi(f,g)(y) = \int_{K}f(x)\hat{g}(xy)dx$ is of Schwartz class $1$, $\Psi(f,g)(y) = \Psi(g,f)(y)$, and for all characters $\chi$ of modulus $1$, $\Psi(f,g)(y)\chi^*(y)$ is uniformly approximated by $\left[\int_{\pi^{m}\got{o}_K\setminus\pi^n\got{o}_K}f(x)\hat{g}(xy)dx\right]\chi^*(y)$, for $m,n\in \Z$ and $m\to -\infty$ and $n\to \infty$. To prove the symmetry of $f$ and $g$, we write
		\[\Psi(f,g)(y) = \int_{K\times K} f(x)g(t)\zeta^{\tr(xty)}dtdx.\]
		
		Since we defined the 2-dimensional integral symmetrically, we have $dtdx = dxdt$, and thus $\Psi(f,g) = \Psi(g,f)$. Moreover, we can show modulo $p^n$, $\Psi(f,g)$ is compactly supported, by approximating the dissipative functions $f$ and $g$ and reduce it to the case $f= \1_{a+\pi^m\got{o}_K}$ and $g= \1_{b+\pi^n\got{o}_K}$, for $a,b\in K$ and $m,n\in \Z$. In this case we have
		\[
		\Psi(f,g)(y) = \zeta^{\tr(aby)}\int_{(x,t)\in \pi^m\got{o}_K\times \pi^n\got{o}_K}\zeta^{\tr(aty+bxy+xty)}dtdx,
		\] 
		
		which is zero when $\ord_\pi(y) < -\delta - \max\{\ord_\pi(a),m\} -\max\{\ord_\pi(b),n\}$, and is therefore compactly supported. As such, we conclude that $\Psi(f,g)$ is dissipative if $f$ and $g$ are. Consequently, as $\hat{f},\hat{g}$ are further assumed to be of Schwartz class $1$, or equivalently $\hat{f}(0) = \hat{g}(0) = 0$, using Proposition \ref{prop.schwartz-criterion} we find that $\Psi(f,g)$ is also of Schwartz class $1$, since $\Psi(f,g)(0) = \hat{g}(0)\int_{K}f(x)dx = 0$. Finally, to show the uniform approximation, we note that as $\chi$ (and thus $\chi^*$) is of modulus $1$, it suffices to show that
		\begin{align*}
			0=&\lim_{\substack{m\to-\infty\\n\to \infty}} 
			\sup_{y\in K}\left|\Psi(f,g)(y) - \int_{\pi^{m}\got{o}_K\setminus\pi^n\got{o}_K} f(x)\hat{g}(xy)dx\right|_p \\
			= &\lim_{\substack{m\to-\infty\\n\to \infty}} 
			\sup_{y\in K}\left|\int_{K\setminus \pi^{m}\got{o}_K} f(x)\hat{g}(xy)dx
			+ \int_{\pi^n\got{o}_K}f(x)\hat{g}(xy)dx
			\right|_p.
		\end{align*}
		
		This is then clear from the following ultrametric estimate for $E\in \{K\setminus \pi^m\got{o}_K, \pi^n\got{o}_K\}$ and the vanishing of $f$ at $0$ and $\infty$:
		\begin{align*}
			\left|\int_{E} f(x)\hat{g}(xy)dx \right|_p \le \sup_{x\in E}|f(x)|_p\cdot  \sup_{x\in K}|\hat{g}(x)|_p.
		\end{align*}
		
		\vspace{3mm}
		
		\textit{Step 4.} Using Corollary \ref{cor.integral expansion} and the facts established in previous steps, we may re-write the product on the left of \eqref{equation.common-zeta-ratio} by:
		\begin{align*}
			\int_K f\chi(x) dx \int_K \hat{g} \chi^*(y) dy 
			&= \sum_{m,n\in \Z}\int_{x\in \pi^m\got{o}_K^\times}\int_{y\in \pi^n\got{o}_K^\times} f(x)\hat{g}(y)\chi(x)\chi^*(y) dydx\\
			&= \sum_{m,n\in \Z}\int_{x\in \pi^m\got{o}_K^\times}\int_{y\in \pi^{n-m}\got{o}_K^\times} f(x)\hat{g}(xy)\chi^*(y) dydx\\
			&= \sum_{m,n\in \Z}
			\int_{y\in \pi^n\got{o}_K^\times}
			\int_{x\in \pi^m\got{o}_K^\times}
			f(x)\hat{g}(xy)\chi^*(y) dxdy\\
			&= \sum_{n\in \Z}
			\int_{y\in \pi^n\got{o}_K^\times}\left(\sum_{m\in\Z}
			\int_{x\in \pi^m\got{o}_K^\times}
			f(x)\hat{g}(xy) dx\right)\chi^*(y)dy\\
			&= \sum_{n\in \Z}\int_{y\in \pi^n\got{o}_K^\times}\left( \int_{K}f(x)\hat{g}(xy)dx\right)\chi^*(y)dy\\
			&= \int_{K}\Psi(f,g)(y)\chi^*(y)dy.
		\end{align*}
		
		As $\Psi(f,g)(y) = \Psi(g,f)(y)$, we conclude that $\int_K f\chi \int_K \hat{g}\chi^* = \int_K \hat{f}\chi^*\int_K g\chi$.
	\end{proof}
	
	\vspace{3mm}
	
	\textit{Remark.} According to Corollary \ref{cor.haar-dissipative-matching} and Theorem \ref{thm.fourier-inversion}, if $f,g$ are dissipative, then one can define the \textit{Cartier pairing}
	\begin{align*}
		\pair{f}{g} = \int_K f(x) \hat{g}(x)dx,
	\end{align*}
	
	which is exactly $\Psi(f,g)(1)$. More generally, if we let $\rho:K^\times\to \aut(\cali{C}(K,\got{o}_p))$ be the regular representation $[\rho(y)f](x) = f(x/y)$, then 
	\begin{align*}
		\Psi(f,g)(y) = \int_K f(x)\hat{g}(xy)dx = \chihaar^{-1}(y)\int_K f(x/y)\hat{g}(x)dx = \chi_{q}(y)\pair{\rho(y)f}{g}.
	\end{align*}
	
	As such, step 3 essentially shows that the Cartier pairing is symmetric and $K^\times$-invariant, and when either $\hat{f}$ or $\hat{g}$ is of Schwartz class 1, the matrix coefficient $\pair{\rho(y)f}{g}$ extends to a function on $K$, and is also of Schwartz class 1.
	
	\vspace{3mm}
	
	\textbf{Examples of Schwartz class functions, and their zeta integrals.}
	
	\vspace{3mm}
	
	In the following, when $I\subseteq \R_+$ is an interval, we say that a function $f$ is of Schwartz class $I$ if $f$ is of Schwartz class $c$ for all $c\in I$. Also, for a multiplicative character $\chi$, recall that we have the decomposition $\chi = \tilde{\chi}\chi_\lambda$ where $\tilde{\chi}$ is the unique character equivalent to $\chi$ with $\tilde{\chi}(\pi) = 1$. 
	\begin{enumerate}
		\item[(a)] We may generalize the computations performed in the appendix. For any $\alpha,\beta \in \C_p^\times$ with $|\alpha|_p<1$ and $|\beta|_p<1$, we define the dissipative functions
		\begin{itemize}
			\item $g_\alpha = \sum_{n\ge 0}\alpha^n \1_{\pi^n\okt}$;
			
			\item $g ^\beta = \sum_{n\ge 0}\beta^n \1_{\pi^{-n-1-\delta}\okt}$;
			
			\item $G_\alpha^\beta = g_\alpha - q^{-1-\delta}\frac{1-\beta q}{1-\alpha q^{-1}} g^\beta$.
		\end{itemize}
		
		Take the Fourier transform, and we get dissipative functions
		\begin{itemize}
			\item $\widehat{g_\alpha} = \frac{q^{-\delta/2}}{1-\alpha q^{-1}}\left[
			(1-\frac{1}{q})\1_{\got{d}^{-1}} -\frac{1-\alpha}{q} g^{\alpha/q}
			\right]$;
			
			\item $\widehat{g^\beta} = \frac{q^{\delta/2}}{1-\beta q}\left[
			(q-1)\1_{\got{o}_K} - q(1-\beta)g_{\beta q}
			\right]$;
			
			\item $\widehat{G_\alpha^\beta} = \frac{q^{-\delta/2}}{1-\alpha q^{-1}}(1-\frac{1}{q})\1_{\got{d}^{-1}\setminus \got{o}_K} 
			+ \frac{q^{-\delta/2}}{1-\alpha q^{-1}}(1-\beta) \left[g_{\beta q} - \frac{1}{q}\frac{1-\alpha}{1-\beta}g^{\alpha/q}\right]$.
		\end{itemize}
		
		We list their Schwartz classes:
		\begin{displaymath}
			\begin{array}{|*{7}{c|}}
				\hline \vrule width 0pt height 2.5ex
				& g_\alpha & g^\beta & G_\alpha^\beta & \widehat{g_\alpha} & \widehat{g^\beta} & \widehat{G_\alpha^\beta} \\
				\hline \vrule width 0pt height 2.5ex
				\text{Schwartz class} & (0,|\alpha|_p^{-1}) & (|\beta|_p,\infty) & (|\beta|_p,|\alpha|_p^{-1}) & (|\alpha|_p,1) & (0,1) & (|\alpha|_p, |\beta|_p^{-1})\\
				\hline
			\end{array}
		\end{displaymath}
		
		As such, we see that both $G_\alpha^\beta$ and $\widehat{G_\alpha^\beta}$ are of Schwartz class $(\max\{|\alpha|_p,|\beta|_p\}, \min\{|\alpha|_p^{-1},|\beta|_p^{-1}\})$. For simplicity, take $G[\alpha] = G_{\alpha q}^{\alpha/q}$, which is of Schwartz class $(|\alpha|_p,|\alpha|_p^{-1})$, and we can compute directly the zeta integrals $Z(G[\alpha],\tilde{\chi}\chi_{\lambda})$ and $Z(\widehat{G[\alpha]},\tilde{\chi}\chi_\lambda)$, where $\lambda \in \C_p^\times $ such that $|\lambda|_p\in (|\alpha|_p,|\alpha|_p^{-1})$:
		\begin{itemize}
			\item If $\tilde{\chi}$ is not the trivial character, both zeta integrals vanish, because $G[\alpha]$ and $\widehat{G[\alpha]}$ are $\got{o}_K^\times$-invariant (as functions on $K$).
			
			\item $Z(G[\alpha],\chi_\lambda) = q^{-\delta/2}(1-\frac{1}{q})\left[\frac{1}{1-\alpha\lambda} - \frac{\lambda^{-1-\delta}}{1-\alpha/\lambda}\right]$.
			
			\item $Z(\widehat{G[\alpha]},\chi_{\lambda}^* = \chi_{q\lambda^{-1}}) = (\lambda q^{-1})^{\delta}(1-\frac{1}{q})\frac{1-\lambda q^{-1}}{1-\lambda^{-1}}\left[\frac{1}{1-\alpha\lambda} - \frac{\lambda^{-1-\delta}}{1-\alpha/\lambda}\right]$.
		\end{itemize}
		
		In particular, we find that $Z(G[\alpha],\chi_\lambda) = q^{\delta/2}\lambda^{-\delta}\frac{1-\lambda^{-1}}{1-\lambda q^{-1}} Z(\widehat{G[\alpha]},\chi_\lambda^*)$, where the factor $q^{\delta/2}\lambda^{-\delta}\frac{1-\lambda^{-1}}{1-\lambda q^{-1}}$ is independent of $\alpha$, as guaranteed by Lemma \ref{lem.local-fe}.
		
		\item[(b)] It is possible to give simpler examples. To begin with, for any multiplicative character $\chi$, define the \textit{level} of $\chi$ to be $0$ if $\chi$ is unramified, and otherwise the smallest positive integer $n$ such that $\chi|_{1+\pi^n\got{o}_K}$ is trivial (when $n=1$ and $\got{o}_K^\times = 1+\pi \got{o}_K$, we say $\chi$ is both of level 0 and 1). Next consider the following functions:
		\begin{itemize}
			\item $h_0 = h_1 = (\1_{1+\pi\got{o}_K} - \frac{1}{q}\1_{\got{o}_K}) - (\1_{\pi+\pi^2\got{o}_K} - \frac{1}{q}\1_{\pi\got{o}_K}) = -\frac{1}{q}\1_{\okt} + \1_{1+\pi\got{o}_K} - \1_{\pi+\pi^2\got{o}_K}$;
			
			\item $h_n = \1_{1+\pi^n\got{o}_K} - \frac{1}{q}\1_{1+\pi^{n-1}\got{o}_K}$ for $n\ge 2$.
		\end{itemize}
		
		Again, we can compute their Fourier transforms:
		\begin{itemize}
			\item $\hat{h}_0(y) = \hat{h}_1(y) = q^{-\delta/2-1}\left[\zeta^{\tr(y)}\1_{\pi^{-\delta-1}\okt} - \frac{1}{q}\zeta^{\tr(\pi y)}\1_{\pi^{-\delta-2}\okt}\right]$;
			
			\item $\hat{h}_n(y) = q^{-\delta/2-n}\zeta^{\tr(y)}\1_{\pi^{-\delta-n}\okt}$ for $n\ge 2$.
		\end{itemize}
		
		We see that all of these functions are of Schwartz class $(0,\infty)$, since they are supported on open compact subsets of $K^\times$. Now, we record their zeta integrals for any multiplicative character $\chi = \tilde{\chi}\chi_\lambda$:
		\begin{displaymath}
			\begin{array}{|*{4}{c|}}
				\hline \vrule width 0pt height 2.5ex
				\text{level of }\chi & 0 & 1& m\ge 2\\
				\hline \vrule width 0pt height 2.5ex
				Z(h_n,\tilde{\chi}\chi_\lambda),\ n=0,1& q^{-\delta/2-2}(1-\lambda)& q^{\delta/2-1}(1-\lambda q^{-1})& 0\\ 
				\hline
				Z(h_n,\tilde{\chi}\chi_\lambda),\ n\ge 2 & 0& 0&
				\begin{array}{c}
					\vrule width 0pt height 2.5ex
					q^{-\delta/2-n} \text{ if }m = n,\\
					\text{otherwise }0
				\end{array}
				\\
				\hline \vrule width 0pt height 2.5ex Z(\hat{h}_n,\tilde{\chi}\chi_\lambda),\ n=0,1& -q^{-1}\lambda^{-\delta-1}(1-\lambda^{-1})& q^{-1}\lambda^{-\delta-1}(1-\lambda^{-1})\got{g}(\tilde{\chi})& 0\\
				\hline
				Z(\hat{h}_n,\tilde{\chi}\chi_\lambda),\ n\ge 2& 0& 0& 
				\begin{array}{c}
					\vrule width 0pt height 2.5ex
					q^{-n}\lambda^{-\delta-n}\got{g}(\tilde{\chi}) \text{ if }m=n,\\
					\text{otherwise }0
				\end{array}\\
				\hline
			\end{array}
		\end{displaymath}
		
		where for $\chi$ of level $n\ge 1$, $\got{g}(\chi) = q^{\delta/2+n}\int_{\okt} \zeta^{\tr(y/\pi^{\delta+n})}\chi(y)dy = \sum_{a\in \okt/1+\pi^n\got{o}_K} \zeta^{\tr(a/\pi^{\delta+n})}\chi(a)$, which is known to be non-zero. In particular, with fixed $\tilde{\chi}$ of level $n$, we have an equality of rigid meromophic functions in the variable $\lambda \in \C_p^\times$:
		\[
		\frac{Z(h_n,\tilde{\chi}\chi_\lambda)}{Z(\hat{h}_n,(\tilde{\chi}\chi_{\lambda})^*)} 
		= 
		\begin{cases*}
			q^{\delta/2}\lambda^{-\delta}\frac{1-\lambda^{-1}}{1-\lambda q^{-1}}, & if $n=0$;\\
			q^{\delta/2+n}\lambda^{-\delta-n}\got{g}(\tilde{\chi}^{-1})^{-1}, & if $n\ge 1$.\\
		\end{cases*}
		\]
	\end{enumerate}
	
	We can now announce the main theorem:
	\begin{thm}
		\label{thm.local-functional-equation}
		Suppose $f$ is a dissipative function and both $f$ and $\hat{f}$ are of Schwartz class $c$ for some $c\ge 1$. Then, for any equivalence class $\got{C}$ of characters parametrized by $\C_p^\times$, the zeta integral $Z(f,\chi)$ can be analytically continued on the Laurent domain $\{\chi\in \got{C}: c^{-1}\le \norm{\chi} \le c\}$, by the functional equation
		\[Z(f,\chi) = \rho(\chi) Z(\hat{f},\chi^*).\]
		
		Here, $\rho(\chi)$ is a meromorphic function on each equivalence class of characters, defined by the functional equation itself when the modulus of $\chi$ is $1$, and for all multiplicative characters by analytic continuation.
	\end{thm}

	\begin{proof}
		For any $n\in \Z_{\ge 0}$ and any multiplicative character $\chi = \tilde{\chi}\chi_\lambda$ of level $n$, we define $\rho(\chi)$ to be $Z(h_n,\chi)/Z(\hat{h}_n,\chi^*)$, as appeared in the second example above. It is clear that $\rho(\chi)$ is rigid meromorphic on each equivalence class of characters, and is entire and non-vanishing unless $n=0$, in which case it has a unique pole at $\lambda =q$ and a unique zero at $\lambda =1$.
		
		\vspace{3mm}
		
		By Corollary \ref{cor.schwartz-filtration}, we know that $f$ and $\hat{f}$ are of Schwartz class $1$, and thus by Lemma \ref{lem.local-fe}, for any character of modulus $1$ and level $n$, we have the functional equation
		\[
		Z(f,\chi) = \frac{Z(h_n,\chi)}{Z(\hat{h}_n,\chi^*)}Z(\hat{f},\chi^*) = \rho(\chi) Z(\hat{f},\chi^*).
		\]
		
		Now fix an equivalence class of multiplicative characters, and let $\chi = \tilde{\chi}\chi_{\lambda}$ only range within this class (namely fix $\tilde{\chi}$ and let $\lambda$ range in $\C_p^\times$). Then, the above equality is an equality between two rigid meromorphic functions on the Laurent domain $\{\lambda\in \C_p^\times: |\lambda|_p = 1\}$. Moreover, by Lemma \ref{lem.analyticity} we know that $Z(\hat{f},\chi^*) = Z(\hat{f},\tilde{\chi}^{-1}\chi_{q\lambda^{-1}})$ is an analytic function on the Laurent domain $\{\lambda\in \C_p^\times,c^{-1}\le |\lambda|_p\le 1\}$. As such, $\rho(\chi)Z(\hat{f},\chi^*)$ is in fact a meromorphic function on the Laurent domain $\{\lambda\in\C_p^\times: c^{-1} \le |\lambda|_p\le 1\}$, extending the analytic function $Z(f,\chi)$ on  $\{\lambda\in\C_p^\times: |\lambda|_p = 1\}$. Therefore, the function $Z(f,\chi)$, which \textit{a priori} is analytic on $\{\lambda\in\C_p^\times: 1\le |\lambda|_p\le c\}$ by Lemma \ref{lem.analyticity}, can be extended to $\{\lambda\in\C_p^\times: c^{-1}\le |\lambda|_p\le c\}$, as was to be shown.
	\end{proof}

	\textit{Remarks.} 
	\begin{enumerate}
		\item We have implicitly used the unicity of analytic continuation for Laurent domains, which we briefly explain here for completeness. For $\alpha_1,\alpha_2,\beta_1,\beta_2\in \C_p^\times$ such that $|\alpha_2|_p\le |\alpha_1|_p \le |\beta_1|_p\le |\beta_2|_p$, we have an inclusion of the Laurent domains $D_1 = \{\lambda\in \C_p^\times: |\alpha_1|_p \le \lambda\le |\beta_1|\} \subseteq D_2 = \{\lambda\in \C_p^\times: |\alpha_2|_p \le \lambda\le |\beta_2|\}$. In this case, it is easy to show that the analytic continuation for any meromorphic function from $D_1$ to $D_2$, if exists, is unique, because the induced map on affinoid algebras $\C_p\chx{\frac{X}{\beta_2},\alpha_2 Y}/(XY-1) \to \C_p\chx{\frac{X}{\beta_1},\alpha_1 Y}/(XY-1)$ is injective.
		
		\item To us it is truly remarkable that we can analytically extend the zeta integrals below modulus $1$. This is all because the Laurent domain $\{\lambda\in \C_p^\times: |\lambda|_p =1\}$ is already (admissible) open in $\C_p^\times$, which is of true $p$-adic nature. The reader should compare this to the classical treatment in §2.4 of \cite{Ta50}, which uses the common region of convergence $0<\text{exponent}<1$ to ensure the complex analytic continuation.
		
		\item As noted in the appendix, our normalization is slightly different from that of Tate. In particular, our $p$-adic $\rho(\chi)$ corresponds to Tate's complex $\rho_{\C}(\chi|\cdot|)$. For example, fix $n\in \Z$ and take $\chi: K^\times \to K^\times$ to be $\chi(x) = q^{-n\ord_\pi(x)}$, which we can regard as both a $p$-adic and complex multiplicative character. In this case, we have the equality (\textit{cf.} p322, \cite{Ta50}):
		\[\rho(\chi) = q^{\delta/2}\chi(\pi)^{-\delta}\frac{1-\chi(\pi)^{-1}}{1-\chi(\pi) q^{-1}} = (q^{\delta})^{(n+1)-1/2}\frac{1-q^{(n+1)-1}}{1-q^{-(n+1)}} =\rho_{\C}(|\cdot|^{n+1})= \rho_{\C}(\chi|\cdot|).\]
		
		\item The factor $\rho(\chi)$ is in fact dependent on $\zeta\in T_\ell(\G_m)$ chosen: if we chose $\zeta^t$ where $t\in \Z_\ell^\times$ then $\hat{f}_{\zeta^t}(r) = \hat{f}_{\zeta}(tr)$, so $\rho_{\zeta^t}(\chi) = \chi(t)\rho_{\zeta}(\chi)$. Note that this dependence is invisible if $\chi$ is unramified; and, when $\chi$ is ramified, is reflected on the dependence of Gauss sum $\got{g}(\chi)$ on the choice of $\zeta$. 
	\end{enumerate}
	
	\section*{Appendix A : An informal approach}
	\addcontentsline{toc}{section}{Appendix A : An informal approach}
	
	In this appendix we will describe an informal approach towards the $p$-adic local functional equation, followed by explicit computations for a certain family of functions. To us, such computations really form the genesis of this note, while the main sections are developed so as to formalize and justify these computations.
	
	\vspace{3mm}
	
	Recall, as discussed in the introduction, the complex local functional equation takes the form:
	\begin{align}
		\tag{A.1}
		\label{equation.lfe-appdx}
		\int_{K}f(x)\chi(x) \frac{dx}{|x|} = \rho_{\C}(\chi) \int_{K}\hat{f}(x)\hat{\chi}(x)\frac{dx}{|x|},
	\end{align}
	
	where $\chi = c|\cdot|^s :K^\times\to \C^\times$ is a continuous character with $0<\re(s)<1$, and $\hat{\chi} = |\cdot|\chi^{-1}$. As such, to ask whether there is a $p$-adic valued analogue, the first thing confronting us is to define relevant objects appearing in \eqref{equation.lfe-appdx} for the $p$-adic setting. More precisely, to formulate \eqref{equation.lfe-appdx} with values in $\C_p$, there are 3 questions we need to answer:
	\begin{enumerate}
		\item[(i)] What is the $p$-adic analogue of $\chi$?
		
		\item[(ii)] What is the $p$-adic analogue of $dx$, and with which can we still do integration?
		
		\item[(iii)] What is the appropriate $p$-adic analogue of Fourier transform, and the subset of functions on which it is defined?
	\end{enumerate}
	
	For simplicity, we shall restrict our attention to the case $K = \Q_\ell$ in this appendix. Firstly, to answer (i), we may tentatively forget about the condition on $\re(s)$, and look at any continuous character from $\Q_\ell^\times$ to $\C_p^\times$. For (ii), we follow the perspective of Riesz-Markov-Kakutani, and think of $dx$ as first a linear functional defined on $\C_p$-valued compactly supported locally constant functions by $dx(\1_{a+\ell^n \Z_\ell}) = \ell^{-n}$, and then extend it by continuity to the completion of this subspace under the sup norm induced by $|\cdot|_p$. Finally we can answer (iii) by na\"ively transplanting the complex-valued Fourier transform on $\Q_\ell$ in the $p$-adic setting, which, at first sight, might appear both artificial and miraculous. Indeed, from complex Fourier transform we have:
	\[
		f = \1_{a+\ell^N\Z_\ell} \longleftrightarrow \hat{f}(t)=\int_{\Q_\ell} f(x) e^{-2\pi i tx} dx 
		= e^{-2\pi iat}\frac{1}{\ell^N}\1_{\ell^{-N}\Z_\ell}(t) ,
	\]
	
	where we recall for $x = \sum_{n\ge h}a_{n}\ell^{n}\in \Q_\ell$, $e^{2\pi ix} = e^{2\pi i \sum_{n=h}^{-1}a_n\ell^{n}}$. As such, it shows $\hat{f}(t)$, \textit{a priori} valued in $\C$, is in fact taking its values in $\Q(\mu_{\ell^\infty}) \subset \C$, because with $t\in \Q_\ell$, $e^{-2\pi i at}$ is an $\ell^{-\ord_\ell(at)}$-th root of unity. Consequently, we may choose once and for all an embedding $\iota_p: \Q(\mu_{\ell^\infty}) \hookrightarrow \C_p$, and define the $p$-adic \textit{mock Fourier transform} $\cali{F}$ as an endomorphism on the completed space from our answer to (ii) by 
	\[\cali{F}(\1_{a+\ell^N\Z_\ell})(x) = \iota_p(e^{-2\pi iax})\frac{1}{\ell^N} \1_{\ell^{-N}\Z_\ell}(x)
	= \frac{1}{\ell^N}\sum_{b\in \Z/\ell^N} \iota_p(e^{-2\pi i\frac{ab}{\ell^N}})\1_{\frac{b}{\ell^N}+\Z_\ell}.\] 
	
	In short, we answered (iii) by interpolating the complex Fourier transform. 
	
	\vspace{3mm}
	
	Now with the three questions resolved, if we assume like \eqref{equation.lfe-appdx} there is also a factor $\rho(\chi)\in \C_p$ independent of $f$ in our $p$-adic setting, we are ready to compute it. As an illustration, we do it explicitly in the unramified case below:
	
	\begin{itemize}
		\item Let $\chi: \Q_\ell^\times \to \C_p^\times$ be an unramified character. Clearly $\chi$ is determined by $\chi(\ell)\in \C_p^\times$, which we denote by $\lambda$.
		
		\item Let $dx$ be the linear functional as above, whose domain is the completion of compactly supported locally constant functions. For example, $dx$ can be evaluated at $\sum_{n\ge 0} p^n \1_{\ell^n \Z_\ell^\times}$, and the evaluation is $(1-1/\ell)\sum_{n\ge 0}p^n\ell^{-n}$.
		
		\item Fix any $\alpha\in \C_p$ such that $|\alpha|_p< \min\{|\lambda|_p,|\lambda|_p^{-1} \}$. Define the continuous function
		\[
		G[\alpha](x) = \sum_{n\ge 0} (\alpha\ell)^{n}\1_{\ell^n \Z_\ell^\times} - \frac{1}{\ell} \sum_{n\ge 0} (\alpha/\ell)^{n}\1_{\ell^{-n-1}\Z_\ell^\times}.
		\]
		
		Then because $\1_{\ell^n \Z_\ell^\times} = \1_{\ell^n \Z_\ell} - \1_{\ell^{n+1}\Z_\ell}$, the mock Fourier transform $\cali{F}$ has the effect: $\cali{F}(\1_{\ell^n \Z_\ell^\times}) = \ell^{-n}\1_{\ell^{-n}\Z_\ell} - \ell^{-n-1}\1_{\ell^{-n-1}\Z_\ell}$. By direct computation, we find: 
		\begin{itemize}
			\item $\cali{F}(\sum_{n\ge 0}(\alpha\ell)^n \1_{\ell^n\Z_\ell^\times}) = \frac{1-1/\ell}{1-\alpha}\1_{\Z_\ell}
			-\frac{1-\alpha\ell}{\ell(1-\alpha)}\sum_{n\ge 0} \alpha^n \1_{\ell^{-n-1}\Z_\ell^\times}$.
			
			\item $\cali{F}(\sum_{n\ge 0}(\alpha/\ell)^n\1_{\ell^{-n-1}\Z_\ell^\times}) = \frac{\ell(1 -1/\ell)}{1-\alpha}\1_{\Z_\ell} - \frac{\ell(1-\alpha/\ell)}{1-\alpha} \sum_{n\ge 0}\alpha^n\1_{\ell^n\Z_\ell^\times}$.
			
			\item Thus $\cali{F}(G[\alpha]) = \frac{1-\alpha/\ell}{1-\alpha}\sum_{n\ge 0}\alpha^n\1_{\ell^n \Z_\ell^\times}
			-\frac{1-\alpha\ell}{\ell(1-\alpha)}\sum_{n\ge 0}\alpha^n\1_{\ell^{-n-1}\Z_\ell^\times}$.
		\end{itemize}
		
		\item We may now compute the two integrals:
		\begin{align*}
			\begin{split}
				\int_{\Q_\ell} G[\alpha](x)\chi(x)dx 
				& = \sum_{n\ge 0}(\alpha\ell)^n\int_{\ell^n\Z_\ell^\times}\chi(x)dx - \frac{1}{\ell}\sum_{n\ge 0}(\alpha/\ell)^n\int_{\ell^{-n-1}\Z_\ell^\times}\chi(x)dx\\
				& = \sum_{n\ge 0}\alpha^n \lambda^n (1-1/\ell) - \frac{1}{\ell}\sum_{n\ge 0}
				\alpha^n \lambda^{-n-1}\ell(1-1/\ell) \\
				& = \frac{1-1/\ell}{1-\alpha\lambda} - \frac{1-1/\ell}{\lambda(1-\alpha\lambda^{-1})}\\
				& = \frac{(1-1/\ell)(1+\alpha)(1-\lambda^{-1})}{(1-\alpha\lambda)(1-\alpha\lambda^{-1})}.
			\end{split}
		\end{align*}
		
		If we let $\chi^* = \chihaar^{-1}\chi^{-1}$ where $\chihaar(x) = 1/\ell^{\ord_\ell(x)}$, then
		\begin{align*}
			\begin{split}
				\int_{\Q_\ell}\cali{F}(G[\alpha])(x)\chi^*(x)dx
				&= \frac{1-\alpha/\ell}{1-\alpha}\sum_{n\ge 0}\alpha^n\int_{\ell^n\Z_\ell^\times}\chi^*(x)dx - \frac{1-\alpha\ell}{\ell(1-\alpha)}\sum_{n\ge 0}\alpha^n\int_{\ell^{-n-1}\Z_\ell^\times}\chi^*(x)dx\\
				&= \frac{1-\alpha/\ell}{1-\alpha}\sum_{n\ge 0} \alpha^n(\ell\lambda^{-1})^n\ell^{-n}(1-1/\ell) - \frac{1-\alpha\ell}{\ell(1-\alpha)}\sum_{n\ge 0}\alpha^n(\ell^{-1}\lambda)^{n+1}\ell^{n+1}(1-1/\ell)\\
				&= \frac{(1-\alpha/\ell)(1-1/\ell)}{(1-\alpha)(1-\alpha\lambda^{-1})}- \frac{(1-\alpha\ell)(1-1/\ell)}{\ell(1-\alpha)}\frac{\lambda}{1-\alpha\lambda}\\
				&= \frac{(1-1/\ell)(1+\alpha)(1-\lambda/\ell)}{(1-\alpha\lambda)(1-\alpha\lambda^{-1})}.
			\end{split}
		\end{align*}
		
		We then conclude that
		\begin{align}
			\tag{A.2}
			\label{equation.local-fe-baby}
			\int_{\Q_\ell} G[\alpha](x)\chi(x) dx = \rho(\chi)
			\int_{\Q_\ell} \cali{F}(G[\alpha])(x)\chi^*(x) dx,
		\end{align}
		
		where $\rho(\chi) = \frac{1-1/\lambda}{1-\lambda/\ell}$.
	\end{itemize}
	
	\textit{Remarks.}
	\begin{enumerate}
		\item Under proper normalization, \eqref{equation.local-fe-baby} takes the same form as its complex counterpart \eqref{equation.lfe-appdx}. In fact, if we take the ``multiplicative Haar measure'' $\textsf{d}x = dx/\chihaar(x)$, then \eqref{equation.local-fe-baby} shows for $f=G[\alpha]$:
		\[
		\int_{\Q_\ell} f\chi \textsf{d}x = \rho(\chi\chihaar^{-1})\int_{\Q_\ell}\hat{f} \chihaar\chi^{-1}\textsf{d}x,
		\] 
		
		which can be formally obtained from \eqref{equation.lfe-appdx} by replacing every complex object by its $p$-adic analogue. We note, however, that our normalization is different from that of \cite{Ta50}, and our $\rho(\chi)$ is really the $p$-adic analogue of the complex factor $\rho_{\C}(\chi|\cdot|)$.
		
		\item We did not take the test function to be $\1_{\Z_\ell}$ as \cite{Ta50} did, for lacking a common region of convergence: On the one hand, $\int_{\Q_\ell}\1_{\Z_\ell} \chi dx$ only converges when $|\lambda|_p<1$ (or $\lambda =1$); and on the other, as $\cali{F}(\1_{\Z_\ell}) = \1_{\Z_\ell}$, the integral $\int_{\Q_\ell}\cali{F}(\1_{\Z_\ell})\chi^* dx$ converges only when $|\chi^*(\ell)|_p = |\ell/\lambda|_p<1$ or $\ell/\lambda = 1$, in which case neither $|\lambda|_p<1$ or $\lambda =1$ could hold true.
		
		\item In the above computations, we have shown that the factor $\rho(\chi)$ is independent of $\alpha$. To prove this fact, we could have avoided the computations if $\lambda\in \bar{\Q}$, by a transcendental method: Fix some transcendental $\tilde{\alpha} \in \C^\times$ such that $|\tilde{\alpha}|_\C < \min\{|\lambda|_{\C},|\lambda^{-1}|_{\C}\}$, and define a complex-valued function $G[\tilde{\alpha}]$ as above by replacing $\alpha$ by $\tilde{\alpha}$. We may then relate the complex setting to the $p$-adic setting by a topological field embedding $\bar{\Q}((\tilde{\alpha})) \hookrightarrow \C_p$ sending $\tilde{\alpha}$ to $\alpha$. The existence of some $\rho(\chi)$ independent of $\alpha$ or $\tilde{\alpha}$ is a corollary of Tate's complex-valued local functional equation. 
		
		\item When we answered question (ii), we ignored the condition on the exponent, namely $0<\re(s)<1$ if the complex character is $c|\cdot|^s$ where $c$ is unitary. The computations above have revealed that the $p$-adic avatar of this condition is a statement about $|\chi(\ell)|_p$. More specifically, with the function $G[\alpha]$ fixed, the local functional equation holds true under the following $p$-adic-natured condition
		\begin{align*}
			|\alpha|_p<|\chi(\ell)|_p<|\alpha|_p^{-1}.
		\end{align*}
		
		This and the next point, in turn, give a preview to Theorem \ref{thm.local-functional-equation}.
		
		\item By letting the parameter $\lambda=\chi(\ell)$ vary, we get a bijection between the rigid analytic variety $\C_p^\times$ and the set of all unramified characters $\Q_\ell^\times \to \C_p^\times$. We then see that the factor $\rho(\chi) = \frac{1-1/\lambda}{1-\lambda/\ell}$, with the variable $\lambda$, is in fact a rigid meromorphic function on $\C_p^\times$. Similarly, when $\alpha$ is fixed, we see that $\int_{\Q_\ell} G[\alpha]\chi dx$ and $\int_{\Q_\ell} \cali{F}(G[\alpha]) \chi^* dx$ are rigid analytic functions on the annulus $\{\lambda \in \C_p^\times: |\alpha|_p^{-1} \le |\lambda|_p \le |\alpha|_p\}$. As such, \eqref{equation.local-fe-baby} becomes an equality between two rigid analytic functions on the annulus, and thus is a bona fide functional equation. This perspective, in turn, concurs well with the one taken in \cite{Ta50}, for which the meromorphic/analytic functions are instead on the Riemann surface $\C/\left(\frac{2\pi i}{\log q}\right)$.
		
		\item We have assumed $\ell\ne p$, but this condition is not essential to the calculations above. The hidden complicacy is that, in the $\ell=p$ case, there is not a well-defined notion of Haar measure, even on the subgroup $\Z_p$. For example, a non-zero bounded $\C_p$-valued measure $\mu$ on $\Z_p$ cannot be invariant under addition, because otherwise for any subset $E$ of non-zero volume, $\mu(p^n E) = p^{-n} \mu(E)$ will tend to infinity when $n$ tends to infinity, and thus unbounded. In another vein, if we want to realize the Haar measure on $\Z_p$ as a locally analytic distribution $\mu$ (i.e., a \textit{distribution continue} as in §1.1.2 of \cite{Co98}) such that $\mu(a+p^n\Z_p) = p^{-n}$ for all $a\in \Z_p$ and $n\in \Z_{\ge 0}$, then the problem is that $\mu$ is not unique: for all $a\in \Z_p$, one can show that, under the Amice transform, the locally analytic distribution $\mu_a$ corresponding to $t^a \frac{\log t}{t-1} \in  \Q_p[[t-1]]$ satisfies this property. That being said, one should note that there is no locally analytic distribution $\mu$ such that $\mu(a+x) = \mu(x)$ for all $a\in \Z_p$, and thus none of the above $\mu_a$ is genuinely addition-invariant.
	\end{enumerate}
	
	\section*{Appendix B : Towards a hypothetical global functional equation}
	
	\addcontentsline{toc}{section}{Appendix B : Towards a hypothetical global functional equation}
	
	Recall that for the complex $L$-function attached to a Dirichlet character $\varphi: \ide_{\Q} \to \bar{\Q}^\times$, we have the global functional equation
	\begin{align}
		\tag{B.1}
		\label{equation.complex-global}
		L(s,\varphi) = \rho(s,\varphi) L(1-s,\varphi)
	\end{align}
	
	for some meromorphic factor $\rho(s,\varphi)$. By the ``theory in the large'' in Tate's thesis, also known as Iwasawa-Tate theory \cite{Iw52}, for a properly normalized multiplicative Haar measure $d^\times a$, \eqref{equation.complex-global} can be re-written as
	\begin{align*}
		\int_{\ide_\Q} f(a)\varphi(a)|a|^s d^\times a = \int_{\ide_\Q} \hat{f}(a)\varphi^{-1}|a|^{1-s}d^\times a.
	\end{align*}
	
	In this appendix, as a continuation of of our methodology, we will record our attempt in seeking a $p$-adic valued global theory, and obtain, via tautology, an idelic integral representation of the $p$-adic polylogarithm. To us, even incomplete, such a theory seems to suggest a potential global functional equation for the $p$-adic zeta function, and thus deserves further investigations.
	
	\subsection*{B.1\quad Idelic characters}
	
	Let $\adefp$ denote the restricted product $\prod_{\ell \ne p}' \Q_\ell$ with respect to $\prod_{\ell \ne p}\Z_\ell$, and similarly $\idefp = \prod_{\ell \ne p}' \Q_\ell^\times$ with respect to $\prod_{\ell \ne p}\Z_\ell^\times$. The group of finite ideles is then given by $\ide_f= \Q_p^\times \times \idefp$. Let also $\Z_{(p)}^+ = \{r\in \Q_{>0}: \gcd(r,p)=1\}$, which, as a group, is isomorphic to $\oplus_{\ell\ne p}\ell^{\Z}$ by unique factorization. We then have a diagonal embedding $\Z_{(p)}^+ \hookrightarrow \idefp$, and it is easy to work out the identification
	\begin{align}
		\tag{B.2}
		\label{equation.idelic-quot}
		\Z_{(p)}^+ \backslash \idefp \xrightarrow{\sim} \prod_{\ell \ne p}\Z_\ell^\times.
	\end{align}
	
	To define characters on $\idefp$, we recall briefly the classical construction of $p$-adically valued idelic characters. Let $\chi_\C: \Q^\times\backslash \ide_{\Q} \to \C^\times$ be a complex idelic character of infinity type $n\in \Z$, namely $\chi_{\C,\infty}(x) = x^{-n}$ for all $x\in \R_{>0}$, where $\chi_{\C,\infty}$ is the restriction of $\chi_{\C}$ on the local component $\R^\times$. The Serre-Tate construction \cite[§1]{Gr81} attaches to $\chi_\C$ a $p$-adic avatar $\chi_\C^{\rm ST}: \Q^\times_{+}\backslash \ide_{f} \to \C_p^\times$, defined by $\chi_\C^{\rm ST}(a) = \chi_\C(a)^{\sigma_p}\cdot a_p^{-n}$, where $a = (a_v)_{v<\infty}\in \ide_{f}$ and $\sigma_p: \bar{\Q}\hookrightarrow \C_p$ is a fixed embedding. Suppose furthermore that $\chi_\C$ is unramified, \textit{i.e.}, $\chi_\C(\prod_{v<\infty}\Z_v^\times)=1$. Then it can be shown that $\chi_\C^{\rm ST}$ is a continuous homomorphism from $\ide_{f}$ to $\C_p^\times$ that enjoys the following two properties:
	\begin{enumerate}
		\item[1)] $\chi_\C^{\rm ST}$ is trivial on the diagonal $\Q^\times_+\subset\ide_{f}$ and on the subgroup $\prod_{\ell \ne p}\Z_\ell^\times$.
		
		\item[2)] The restriction of $\chi_\C^{\rm ST}$ to the local component $\Q_p^\times$ is trivial on the cyclic subgroup $p^\Z$, and there exists an integer $n$ (the infinity type of $\chi_\C$) such that $\chi_\C^{\rm ST}$ is of the form $x\mapsto x^{-n}$ for $x\in \Z_p^\times$.
	\end{enumerate}
	
	The two properties in fact characterize the image under the Serre-Tate construction of unramified characters, but we will not need this fact. We emphasize that the true merit of this characterization is that it frees us from archimedean considerations. Mimicking this, we introduce
	
	\begin{de}
		A continuous homomorphism $\chi^p: \idefp\to \C_p^\times$ will be referred to as an unramified character on $\idefp$ if it is the restriction of some continuous homomorphism $\chi: \ide_f \to \C_p^\times$ that enjoys the following two properties:
		\begin{enumerate}
			\item[i)] $\chi$ is trivial on the diagonal $\Q^\times_{+}\subset \ide_f$ and on the subgroup $\prod_{\ell \ne p}\Z_\ell^\times$.
			
			\item[ii)] The restriction of $\chi$ to the local component $\Q_p^\times$ is trivial on the cyclic subgroup $p^\Z$.
		\end{enumerate}
		
		In what follows, for a finite place $v$ we denote by $\chi_v$ the local character $\chi_v: \Q_v^\times\hookrightarrow \ide_{f}\xrightarrow{\chi} \C_p^\times$.
	\end{de} 
	
	\textit{Remarks.}
	\begin{enumerate}
		\item There is at most one character $\chi: \ide_{f}\to \C_p^\times$ satisfying i) and ii) that restricts to a given $\chi^p:\idefp\to \C_p^\times$. This is because $\chi = \chi_p\chi^p$, and properties i) and ii) show that the restriction of $\chi_p$ on the dense subset $p^{\Z}\Z_{(p)}^+\subset \Q_p^\times$ is determined by $\chi^p$, and thus $\chi_p$ is determined by $\chi^p$.
		
		\item Property ii) is redundant, since i) implies that $\chi_p(p)\prod_{\ell \ne p}\chi_\ell(p) = 1$ and $\chi_\ell(p) = 1$ for all $\ell\ne p$.
	\end{enumerate}
	
	From now on we will often drop the adjective ``unramified'' since it is the only scenario under consideration. Also, when there is no confusion, we shall not distinguish the character $\chi^p: \idefp\to \C_p^\times$ from its lift $\chi: \ide_{f}\to \C_p^\times$, thanks to the first remark.
	
	\begin{prop}
		There is a bijection from the set of unramified characters on $\idefp$ to continuous homomorphisms $\lambda_p: \Z_p^\times \to \got{o}_p^\times$; the map is given by the restriction $\chi \mapsto \chi|_{\Z_p^\times}$.
	\end{prop}
	
	\begin{proof}
		Suppose $\chi^p$ is an unramified character on $\idefp$. That $\lambda_p = \chi|_{\Z_p^\times}$ is determined by $\chi^p$ follows from the first remark above. Conversely, if $\lambda_p$ is a character $\Z_p^\times \to \got{o}_p^\times$, then we may define $\chi_p$ to be the unique extension of $\lambda_p$ to $\Q_p^\times$ trivial on $p^\Z$, and $\chi_\ell$ to be the unramified local character with $\chi_\ell(\ell) = \chi_p(\ell)^{-1}$. In turn, by taking $\chi^p = \prod_{\ell\ne p}\chi_\ell$ and $\chi = \chi_p\chi^p$ we have a desired unramified character on $\idefp$.
	\end{proof}
	
	\textbf{Examples.} Thus to construct a character on $\idefp$ it suffices to construct one on $\Z_p^\times$. We then have the \textit{principal} cyclotomic character $\kappa: \Z_p^\times \twoheadrightarrow 1+p\Z_p \hookrightarrow \got{o}
	_p^\times$, of which the character on $\idefp$ is such that $\kappa(n\prod_{\ell\ne p}\Z_\ell^\times) = \chx{n}^{-1}$ for all $n\in\Z_{(p)}^+$. Also, for any Dirichlet character $\psi: (\Z/p^t)^\times \to \bar{\Q}^\times$, we may regard it as a character on $\Z_p^\times$, and thus on $\idefp$; its evaluation at the coset $n\prod_{\ell \ne p}\Z_\ell^\times$ for $n\in \Z_{(p)}^+$ is a constant $\psi(n)^{-1}$. We note that the usual cyclotomic character, $\kappa\omega$, coincides with the absolute value $|\cdot|: \idefp \to \C_p^\times$, $x = (x_\ell)_{\ell\ne p} \mapsto \prod_{\ell\ne p} |x_\ell|_\ell$; here for $x_\ell\in \Q_\ell^\times$, $|x_\ell|_\ell = \ell^{-{\mathrm{ord}_\ell(x_\ell)}}$. In the following we shall restrict exclusively to characters on $\idefp$ of the form $\chi = \kappa^s \psi$ for some $s\in \Z_p$ and $\psi$ as above.
	
	\subsection*{B.2\quad Global theory of integration}
	
	Let $\cali{C}(\adefp,\got{o}_p)$ be the set of continuous functions from $\adefp$ to $\got{o}_p$. In the same spirit as §2, we have the following topological linearization:
	\begin{align*}
		\textstyle
		\cali{C}(\adefp,\got{o}_p) = \plim_{n} \plim_{N} \ilim_{M} \cali{C}(N^{-1}\prod_{\ell\ne p} \Z_\ell /M\prod_{\ell\ne p}\Z_\ell, \got{o}_p/p^n) = \plim_{n} \plim_{N} \ilim_{M} \cali{C}(N^{-1}\Z/M\Z, \got{o}_p/p^n).
	\end{align*}
	
	Next, we define $\cali{D}(\adefp,\got{o}_p) = \Hom^{\rm cts}_{\got{o}_p}(\cali{C}(\adefp,\got{o}_p),\got{o}_p)$. Again by the Cartier duality formalism, it can be shown that there is an embedding (coarse Fourier transform)
	\begin{align*}
		\cali{D}(\adefp,\got{o}_p) \simeq \plim_{n}\ilim_N \plim_M \cali{C}(M^{-1}\Z/N\Z, \got{o}_p/p^n) \hookrightarrow \cali{C}(\adefp,\got{o}_p).
	\end{align*}
	
	As usual, the explicit formula is
	\begin{align*}
		\mu \mapsto \hat{\mu}(b) = \int_{\adefp} \zeta^{ba} \mu(a)\quad\text{ for all }b\in \adefp.
	\end{align*}
	
	To introduce functions with desired decay, we consider $\cali{C}(\idefp,\got{o}_p)$, the space of continuous functions from $\idefp$ to $\got{o}_p$, which has a dense subspace $\cali{S}(\idefp,\op)$ consisting of compactly supported locally constant functions. On $\cali{S}(\idefp,\op)$ we have a linear functional $da$ whose evaluation at $\1_{\prod_{\ell \ne p}\Z_\ell^\times}$ is $1$, and satisfies the property $d(ta) = |t|da$ for all $t\in \idefp$. Furthermore, by a \textit{Schwartz class} function we mean a function $g\in \cali{C}(\idefp,\got{o}_p)$ such that for all $s\in \Z_p$ and all characters $\psi$, $g\kappa^s \psi da$ is in $\cali{D}(\adefp,\got{o}_p)$, where the pairing with any $f\in \cali{C}(\adefp,\got{o}_p)$ is via the following procedure:
	\[
	\begin{tikzcd}
		\cali{C}(\adefp,\got{o}_p) \ar[r,"\text{res}"] & \cali{C}(\idefp,\got{o}_p) \ar[r,"g\kappa^s \psi da"] &
		\got{o}_p\\
		f \ar[rr,mapsto]&&\int_{\idefp} f(a)g(a)\kappa^s\psi(a) da,
	\end{tikzcd}
	\]
	
	where $\text{res}: \cali{C}(\adefp,\got{o}_p) \to \cali{C}(\idefp,\got{o}_p)$ is the restriction map (valid and continuous because $\idefp$ has a finer topology than the inherited one from $\adefp$).
	
	\vspace{3mm}
	
	\textit{Remark.} Contrary to the local case, we do not have a well-behaved global multiplicative theory. More precisely, we see that there is no canonical section to the restriction map $\cali{C}(\adefp,\got{o}_p) \to \cali{C}(\idefp,\got{o}_p)$ since the topology of $\idefp$ is strictly finer than that of $\adefp$; while recall in the local case we essentially used the canonical section $\cali{C}(\Q_\ell^\times,\got{o}_p) \to \cali{C}(\Q_\ell,\got{o}_p)$ of extending by zero to define Schwartz class functions. In turn, this becomes problematic, as when $f\in \cali{C}(\idefp,\got{o}_p)$ is a Schwartz class function, $\hat{f}$ will be a function on $\adefp$, and is thus never of Schwartz class under our definition!
	
	\vspace{3mm}
	
	\textbf{Example.} The function $\1_{\prod_{\ell\ne p}\Z_\ell^\times}$ is of Schwartz class, because for any unramified character $\chi$ on $\idefp$ we have $\chi\1_{\prod_{\ell \ne p}\Z_\ell^\times} = \1_{\prod_{\ell \ne p}\Z_\ell^\times}$. Now let us compute its Fourier transform, which is given by
	\begin{align*}
		(\1_{\prod_{\ell\ne p}\Z_\ell^\times})^\wedge(b) = \int_{\prod_{\ell\ne p}\Z_\ell^\times} \zeta^{ba} da.
	\end{align*}
	
	Since for any $b\in \adefp$, its $\ell$-component is inside $\Z_\ell$ for all but finitely many $\ell$, we have 
	\begin{align*}
		\int_{\prod_{\ell\ne p}\Z_\ell^\times} \zeta^{ba}da = \lim_{S} \int_{\prod_{\ell\in S}\Z_\ell^\times} \zeta^{bx} \frac{dx}{\prod_{\ell\in S}(1-1/\ell)},
	\end{align*}
	
	where the limit is over all finite sets of primes not containing $p$, and $dx$ is the product of the local Haar measures defined in §2.3. We thus have
	\begin{align*}
		(\1_{\prod_{\ell\ne p}\Z_\ell^\times})^\wedge = \prod_{\ell\ne p} \frac{1}{1-1/\ell}(\1_{\Z_\ell^\times})^\wedge = \prod_{\ell\ne p} \left(\1_{\Z_\ell} + \frac{1}{1-\ell}\1_{\ell^{-1}\Z_\ell^\times}\right),
	\end{align*}
	
	which clearly is a uniformly continuous function on $\adefp$.
	
	\subsection*{B.3\quad Conjectural functional equation}
	
	We are now ready to indicate why the global integral is interesting. For $\alpha \in \C_p$ with $|\alpha|_p<1$, we consider the Schwartz class function $E_\alpha = \sum_{n\ge 1, p\nmid n} \alpha^n \1_{n\prod_{\ell \ne p}\Z_\ell^\times}$. Let also $\chi = (\kappa^{1-s}\psi\omega)^{-1}$ be a character on $\idefp$, with $s\in \Z_p$ and $\psi$ a Dirichlet character of a $p$-power conductor. Then we have the following integral representation of a (twisted) polylogarithm:
	\begin{align}
		\tag{B.3}
		\label{equation.integral-representation-idelic}
		\begin{split}
			\int_{\idefp} \chi(a)E_\alpha(a)da &= \sum_{n\ge 1, p\nmid n}\alpha^n \int_{\idefp}\chi(a)\1_{n\prod\Z_\ell^\times}(a) da\\
			&= \sum_{n\ge 1, p\nmid n} \alpha^n\chx{n}^{1-s}\psi\omega(n)\frac{1}{n}\\
			&= \sum_{n\ge 1, p\nmid n}\frac{\alpha^n\psi(n)}{\chx{n}^s}.
		\end{split}
	\end{align}
	
	Denote the last function by $\cali{L}(\alpha,s,\psi)$. By an observation of Koblitz \cite[paragraph following Lemma 1]{Ko79} and Coleman \cite[Proposition 6.2]{Co82}, for fixed $s$ and $\psi$, the function $\cali{L}(\alpha,s,\psi)$ is rigid analytic on $\mathbf{P}^1(\C_p) - \{z\in\C_p: |z-1|_p<1\}$. When $s=k\ge 2$ is an integer, Coleman \cite[Corollary 7.1a]{Co82} essentially showed that $\cali{L}(\alpha,k,\psi)$ can be defined on the removed disc $\{z\in\C_p: |z-1|_p<1\}$ so that it is continuous on $\mathbf{P}^1(\Q_p)$; moreover, the following identity holds:
	\begin{align*}
		\cali{L}(1,k,\psi) = L_p(k,\psi\omega),
	\end{align*}
	
	where the right hand side is the Kubota-Leopoldt $p$-adic $L$-function attached to the character $\psi\omega$. Alternatively, when $p\ne 2$ we may allow $s\in \Z_p$ to be arbitrary, and evaluate the polylogrithm $\cali{L}(\alpha,s,\psi)$ at $\alpha=-1$ to deduce
	\begin{align*}
		\cali{L}(-1,s,\psi) = \frac{1}{2}\sum_{n\ge 1} \sum_{p^{n-1}\le m<p^n,p\nmid m} \frac{(-1)^m\psi(m)}{\chx{m}^s} = -(1-\psi\omega(2)\chx{2}^{1-s})L_p(s,\psi\omega),
	\end{align*}
	
	where the last equality is in essence due to Delbourgo \cite[Theorem 1.1]{De06}.
	
	\vspace{3mm}
	
	To proceed further, may we be forgiven to make the following bold speculations: We expect that it should be possible to give a precise meaning to the integral $\int_{\idefp} \chi(a)\hat{E}_{\alpha}(a)da$ for any $\alpha\in \C_p$ with $|\alpha|_p<1$ and for any $\chi = \kappa^s\psi$; and there shall be a constant $\rho(\chi)\in \C_p^\times$ independent of $\alpha$ such that
	\begin{align*}
		\int_{\idefp} \chi(a)E_{\alpha}(a)da = \rho(\chi)\int_{\idefp} (\kappa\omega\chi)^{-1}(a)\hat{E}_{\alpha}(a) da.
	\end{align*}
	
	Granted this, take $\chi=(\kappa^{1-s}\psi\omega)^{-1}$ and let $\alpha$ tend to $1$, and we would obtain a conjectural functional equation of the Kubota-Leopoldt $p$-adic $L$-function:
	\begin{align}
		\tag{B.4}
		\label{equation.global-functional}
		L_p(s,\psi\omega) = \rho(\chi) L_p(1-s,\psi^{-1}).
	\end{align}
	
	We stress that the conjectural functional equation \eqref{equation.global-functional} has important consequences, since it can relate the values of the $p$-adic $L$-function at positive integers, which are at present largely mysterious, to those at negative integers, which are known to be Bernoulli numbers multiplied by the Euler factors at $p$. In particular, using it one should be able to tell whether or not $L_p(k,\psi\omega^{k-1})$ vanishes for any positive integer $k\ge 2$, which in turn is closely tied to whether some Iwasawa modules are of finite cardinality (for a nice summary see §3.3 of \cite{Co15}).
	
	\vspace{3mm}
	
	\textit{Final Remark.} Strictly speaking, our discussion above have only indicated the limit on the left hand side is valid when $s\ge 2$ is an integer. On the other hand, the limit on the right is heuristically obtained from \eqref{equation.integral-representation-idelic} by treating $\hat{E}_\alpha$ as if as $E_\alpha$, and replacing $\chi$ by $(|\cdot|\chi)^{-1} = (\kappa\omega\chi)^{-1}$. The surrogation of $E_\alpha$ for $\hat{E}_\alpha$ might eventually become legit, as we imagine that $\lim_{\alpha\to 1} E_\alpha$``$=$''$\1_{\prod_{\ell \ne p}\Z_\ell}$, and $\1_{\prod_{\ell \ne p}\Z_\ell}$ should be equal to its own Fourier transform. To us, the main difficulty is to find an appropriate way to precisely formulate the right hand side of the functional equation, where the method of $p$-adic analysis appears to have reached its limit. In this regard, we are curious whether we could have developed everything purely in terms of group schemes, as contrary to mixing it with $p$-adic analysis. We plan to study these questions in the near future.

	\paragraph{Acknowledgement.} The author is very thankful to Antonio Lei, for his constant encouragement and for reading the draft and providing many helpful suggestions on the organization and the grammar; to Yiannis Sakellaridis for the many inspirational discussions in Newark. We would also like to thank Professor John Coates for kindly pointing out an important application of a potential global functional equation in Iwasawa theory. Finally, we thank the referee for their careful reading and helpful suggestions.
	
	\bibliographystyle{alpha}
	\bibliography{fourier_v2}
	
	\vspace{5mm}
	
	\textsc{Department of Mathematics, Johns Hopkins University, 404 Krieger Hall, 3400 N. Charles Street, Baltimore, MD 21218, USA}
	
	\textit{Email address: }\texttt{lzhao39@jhu.edu}
\end{document}